\documentclass[12pt]{amsart}
\usepackage[plainpages=false]{hyperref}
\usepackage{amsfonts,latexsym,rawfonts,amsmath,amssymb,amsthm,mathrsfs}
\usepackage{amsmath,amssymb,amsfonts,latexsym,lscape,rawfonts}
\usepackage[usenames]{color}
\usepackage{graphicx}

\DeclareMathAlphabet\oldmathcal{OMS}        {cmsy}{b}{n}
\SetMathAlphabet    \oldmathcal{normal}{OMS}{cmsy}{m}{n}
\DeclareMathAlphabet\oldmathbcal{OMS}       {cmsy}{b}{n}
 
\usepackage{eucal}

\usepackage[active]{srcltx}

\def\TT{{\mathbb{T}}}

\def\PP{{\mathbb{P}}} 

\def\ZZ{{\mathbb{Z}}}

\def\CC{{\mathbb{C}}}
\def\RR{{\mathbb{R}}} 
 
\def\TT{{\mathbb{T}}}

\newtheorem{thm}{Theorem}[section]

\newtheorem{question}[thm]{Question}
\newtheorem{cor}[thm]{Corollary}
\newtheorem{lemma}[thm]{Lemma}

\newtheorem{prop}[thm]{Proposition}

\newtheorem{Def}[thm]{Definition}

\newtheorem{def/prop}[thm]{Definition/Proposition}

\newtheorem*{ack}{Acknowledgements}

\newenvironment{example}{\medskip \refstepcounter{theorem}
\noindent  {\bf Example \thetheorem}.\rm}{\,}

\def\<{\langle}

\def\>{\rangle}

\def\BOne{{\mathchoice {\rm 1\mskip-4mu l} {\rm 1\mskip-4mu l}
                          {\rm 1\mskip-4.5mu l} {\rm 1\mskip-5mu l}}}

\def\fract#1#2{\raise4pt\hbox{$ #1 \atop #2 $}}
\def\decdnar#1{\phantom{\hbox{$\scriptstyle{#1}$}}
\left\downarrow\vbox{\vskip15pt\hbox{$\scriptstyle{#1}$}}\right.}

\def\bbc{{\mathbb C}}

\def\bbp{{\mathbb P}}
\def\bbq{{\mathbb Q}}
\def\bbr{{\mathbb R}}
\def\bbs{{\mathbb S}}
\def\bbt{{\mathbb T}}

\def\bbz{{\mathbb Z}}

\def\gra{\alpha}
\def\grb{\beta}

\def\gre{\epsilon}

\def\grk{\kappa}

\def\gro{\omega}

\def\grr{\rho}

\def\grt{\tau}

\def\grD{\Delta}

\def\grG{\Gamma}

\def\grO{\Omega}

\def\grS{\Sigma}

\def\bfb{{\bf b}}

\def\bfu{{\bf u}}
\def\bfv{{\bf v}}
\def\bfw{{\bf w}}

\def\cala{{\mathcal A}}
\def\calb{{\mathcal B}}

\def\calo{{\mathcal O}}

\def\cald{{\mathcal D}}
\def\cale{{\mathcal E}}
\def\calf{{\mathcal F}}

\def\calh{{\mathcal H}}

\def\call{{\mathcal L}}
\def\calm{{\mathcal M}}

\def\calo{{\mathcal O}}

\def\calr{{\mathcal R}}
\def\cals{{\oldmathcal S}}

\def\calv{{\mathcal V}}

\def\la#1{\hbox to #1pc{\leftarrowfill}}
\def\ra#1{\hbox to #1pc{\rightarrowfill}}
\def\calz{{\oldmathcal Z}}

\def\X{\frak{X}}

\def\ga{{\mathfrak a}}
\def\gb{{\mathfrak b}}

\def\gf{{\mathfrak f}}
\def\gg{{\mathfrak g}}
\def\gh{{\mathfrak h}}

\def\gl{{\mathfrak l}}

\def\gn{{\mathfrak n}}
\def\go{{\mathfrak o}}

\def\gr{{\mathfrak r}}
\def\gs{{\mathfrak s}}
\def\gt{{\mathfrak t}}
\def\gu{{\mathfrak u}}

\def\gA{{\mathfrak A}}
\def\gB{{\mathfrak B}}
\def\gC{{\mathfrak C}}

\def\gF{{\mathfrak F}}

\def\gM{{\mathfrak M}}

\def\hook{\mathbin{\hbox to 6pt{%
                 \vrule height0.4pt width5pt depth0pt
                 \kern-.4pt
                 \vrule height6pt width0.4pt depth0pt\hss}}}

\def\R{\mathbb R}\def\C{\mathbb C}\def\Z{\mathbb Z} \def\Q{\mathbb Q}  
\def\bT{\mathbb T}
\def\bu{{\bf n}}

\def\kt{\mathfrak{t}}

\def\rar{\rightarrow}
\def\Span{\mbox{span}}
\def\im{\mbox{Im}}

 \newcommand*{\quot}[2]%
{\ensuremath{%
   \raisebox{.35ex}{\ensuremath{#1}}\big/\raisebox{-.35ex}{\ensuremath{#2}}}}

 \theoremstyle{remark}
 \newtheorem{rem}[thm]{Remark}

\title{Reducibility in Sasakian Geometry} 
\author{Charles P. Boyer}
\author{ Hongnian Huang}
\author{Eveline Legendre}
\author{Christina W. T{\o}nnesen-Friedman}

\date{\today}
\address{Charles P. Boyer, Department of Mathematics and Statistics,
University of New Mexico, Albuquerque, NM 87131.}
\email{cboyer@math.unm.edu} 
\address{Hongnian Huang, Department of Mathematics and Statistics,
University of New Mexico, Albuquerque, NM 87131.}
\email{hnhuang@gmail.com} 
\address{Eveline Legendre\\ Universit\'e Paul Sabatier\\
Institut de Math\'ematiques de Toulouse\\ 118 route de Narbonne\\
31062 Toulouse\\ France}
\email{eveline.legendre@math.univ-toulouse.fr}
\address{Christina W. T{\o}nnesen-Friedman, Department of Mathematics, Union
College, Schenectady, New York 12308, USA } \email{tonnesec@union.edu}

\thanks{The first author was partially supported by a grant (\#245002) from the Simons Foundation. The third author is partially supported by France ANR project EMARKS No ANR-14-CE25-0010. The fourth author was partially supported by grant \#208799 from the Simons Foundation.}

\begin{document}

\maketitle

\markboth{Reducibility in Sasakian Geometry}{C. Boyer, H. Huang, E. Legendre and C. T{\o}nnesen-Friedman}

\begin{abstract}
The purpose of this paper is to study reducibility properties in Sasakian geometry. First we give the Sasaki version of the de Rham Decomposition Theorem; however, we need a mild technical assumption on the Sasaki automorphism group which includes the toric case. Next we introduce the concept of {\it cone reducible} and consider $S^3$ bundles over a smooth projective algebraic variety where we give a classification result concerning contact structures admitting the action of a 2-torus of Reeb type. In particular, we can classify all such Sasakian structures up to contact isotopy on $S^3$ bundles over a Riemann surface of genus greater than zero. Finally, we show that in the toric case an extremal Sasaki metric on a Sasaki join always splits.
\end{abstract}

\tableofcontents

\section{Introduction}

A fundamental result in Riemannian geometry is the de Rham decomposition theorem stating that any simply connected complete Riemannian manifold with reducible holonomy is the product of simply connected complete Riemannian manifolds. Of course, the product cannot make sense in the Sasaki category because a product of Sasaki manifolds cannot be Sasaki for dimensional reasons. However, we can consider product structures that are transverse to the characteristic foliation of the Reeb vector field. This leads directly to the concept of a reducible Sasakian structure which is the subject of this paper.

This notion of reducibility in Sasakian geometry was first discussed in the context of the ``join construction'' for quasi-regular Sasaki-Einstein manifolds in \cite{BG00a}, and later developed in the general Sasakian context in \cite{BGO06}. However, recently a more intrinsic definition of reducibility on the tangent level was given by He and Sun \cite{HeSu12b}, and following the Riemannian analog we shall refer to the join as {\it decomposable}.

Given two quasi-regular Sasaki manifolds $M_1, M_2$ with Reeb vector fields $\xi_1,\xi_2$, respectively and a pair of relatively prime integers $(l_1, l_2)$, one can construct a new quasi-regular Sasaki manifold $M_3= M_1 \star_{l_1, l_2} M_2$ whose Reeb vector field $\xi_3$ is a certain linear combination of $\xi_1$ and $\xi_2$ and whose underlying manifold is the quotient of the product $M_1\times M_2$ by an $S^1$ action depending on $l_1$ and $l_2$. This construction is known as the {\it join construction} and is the analogue of the product in Sasakian geometry, that is the closest notion one can define in this category. Let $S^1_\xi$ denote the $S^1$ action generated by the Reeb vector field $\xi$. In particular, there is a K\"ahler isometry between the K\"ahler orbifolds $(\quot{M_3}{S^1_{\xi_{3}}},\gro_3)$ and $(\quot{M_1}{S^1_{\xi_{1}}} \times \quot{M_2}{S^1_{\xi_{2}}},l_1\gro_1+l_2\gro_2)$.

The purpose of this paper is to study further the concept of reducible Sasakian structures and the join of Sasaki manifolds. We address three main issues concerning this notion.

\subsection{Are Simply Connected Reducible Sasaki Manifolds always a Join?}
This is essentially Question 1.1 of He-Sun \cite{HeSu12b} and the theme of Section \ref{joinredsect}.
As in the product operation, the resulting structure has a special splitting or {\it reducible} property in terms of holonomy. However, on a Sasaki manifold we have to consider the so-called `transverse holonomy' that is the holonomy of the metric $g^T$ that is transverse to the characteristic foliation $\calf_\xi$ generated by the Reeb vector field $\xi$. The concept of a transverse decomposable Sasakian structure is that of the join as mentioned previously. Contrary to the Riemannian or K\"ahlerian case, this reduciblility property is not quite as obvious in the Sasakian case. 

He and Sun \cite{HeSu12b} gave a definition of reducible Sasaki manifold (see definition \ref{HSdef}) and proved that a reducible Sasaki manifold whose irreducible components have either transverse positive Ricci curvature or a transverse flat structure must be quasi-regular. That is, irregular Sasakian structures are irreducible in this case. 

Here we extend the He-Sun result to more general cases which include the toric case. However, at this stage we fall short of a complete answer to the question. Our first main result is:

\begin{thm}\label{theoREDimpliesREG} Let $\cals=(\xi,\eta,\Phi,g)$ be a reducible Sasakian structure on a compact connected manifold $M$. Assume further that each piece (not necessarily irreducible) has Sasaki automorphism groups of dimension greater than one. Then $\xi$ is quasi-regular. Furthermore, if $M$ is simply connected, $\cals$ is the join of quasi-regular compact Sasaki manifolds with Sasaki automorphism groups of dimension greater than one. Equivalently $\cals$ is decomposable.  
 \end{thm}

As a particular case we have

\begin{cor}\label{toricred}
An irregular toric Sasaki manifold is irreducible.
\end{cor}

Even when one piece has a one dimensional Sasaki automorphism group we can say more if that piece also has either positive transverse Ricci curvature or is transversally flat by incorporating the He-Sun result with ours using Molino theory. This is Proposition \ref{poss1} below.

\subsection{When is a Sasaki Manifold Cone Reducible?}
In Section \ref{coneredsect} we introduce the concepts of {\it cone reducible (decomposable)} in which case the given Sasakian structure is not reducible (decomposable); however, one can obtain a reducible (decomposable) Sasakian structure after deforming in the Sasaki cone. We concentrate our attention mainly, though not entirely, on $S^3$ bundles over a smooth compact projective algebraic variety. The reason for this is two-fold. First many such examples of this occur in \cite{BoTo13,BoTo14a}, and second a well known theorem of Eliashberg \cite{Eli92} says that there is only one tight contact structure on $S^3$ up to oriented isotopy, and that is the standard contact structure. This allows us to give certain classification results.

\begin{question}
Which Sasakian structures can be obtained from the join construction allowing for deformations of the Reeb vector field in the Sasaki cone?
\end{question}

Such a Sasakian structure is said to be {\it cone decomposable}. See Definition \ref{conered} below.

\begin{thm}\label{s3riemcor}
Let $M$ be an $S^3$-bundle over a compact smooth algebraic variety of the form $N=N'\times A(N)$ where $N'$ has finite automorphism group and $A(N)$ is the Albanese variety of $N$, and let $\cald$ be a co-oriented contact structure on $M$ with an effective $\bbt^2$ action of Reeb type. Then $\bbt^2$ acts trivially on all $N$. Suppose  further that the Picard number $\grr(N)=1$, then $(M,\cald)$ is of Sasaki type and there is an underlying $\bbt^2$ invariant Sasaki CR structure $(\cald,J)$ that is cone decomposable, i.e. there is a Sasakian structure $\cals\in\gt^+_2$ that is isomorphic to an $S^3_\bfw$-join construction of \cite{BoTo14a}. 
\end{thm}

There are also choices of Sasaki CR structures that are cone reducible, but not cone decomposable. An example is given by Example \ref{rulsurfex} below. See also Remark \ref{nneq0rem}. It also follows from the orbifold version of the de Rham decomposition, Lemma \ref{orbdeRh}, that any simply connected cone reducible Sasaki manifold is cone decomposable.

A particular case of interest where we obtain stronger results occurs when $N$ is a Riemann surface of genus $g$ studied in \cite{BoTo13} in which case the Picard number is automatically one. We recall that the Picard group of line bundles on $\grS_g$ has the form ${\rm Pic}(\grS_g)\approx T^{2g}\times \bbz$ where $T^{2g}$ is the Jacobian torus. However, for us line bundles of degree $-n$ are equivalent to those of degree $n$, so we are interested in ${\rm Pic}(\grS_g)/\bbz_2\approx T^{2g}\times \bbz_{\geq 0}$. Here we let $\gM_g$ denote the moduli space of complex structures on $\grS_g$.

\begin{thm}\label{Riesurfcor}
Let $M^5$ be an $S^3$-bundle over a Riemann surface $\grS_g$ of genus $g>0$ with a contact structure $\cald$ that admits an action of a $2$-torus $\bbt^2$ of Reeb type. Then each compatible $\bbt^2$-invariant CR structure $(\cald,J)$ is of Sasaki type and cone reducible. Moreover, if $g>1$ for each fixed $\grt\in \gM_g(\grS_g)$, there is a one-to-one correspondence between such CR structures and elements of ${\rm Pic}(\grS_g)/\bbz_2\approx T^{2g}\times\bbz_{\geq 0}$ and for each degree $n\in\bbz$ there is precisely one choice of CR structure $(\cald,J)$ that is cone decomposable. If $g=1$ and the degree of the line bundle is zero, then for each fixed $\grt\in \gM_g(\grS_1)$ there is a one-to-one correspondence between cone reducible $\bbt^2$ invariant CR structures and points of $\bbc\bbp^1$; whereas, if the line bundle has degree $n> 0$, then up to biholomorphism there is precisely one $\bbt^2$ invariant CR structure and it is cone decomposable.
\end{thm}

We remark that Theorem \ref{Riesurfcor} does not hold in the genus $g=0$ case. Indeed, Example \ref{BPlink} below gives a $\bbt^2$ invariant CR Sasakian structure on $S^2\times S^3$ that is cone irreducible. For these Sasakian structure, the maximal torus (in particular, the reduced Sasaki cone) is $2$-dimensional. The case of a $\bbt^3$ torus acting on $S^2\times S^3$ becomes toric and then it is cone reducible as follows from \cite{BoPa10} and is a special case of Theorem 1.6 below.

Recall that there are precisely two different diffeomorphism types of $S^3$ bundles over Riemann surfaces, the trivial bundle $\grS_g\times S^3$ and the non-trivial bundle $\grS_g\tilde{\times}S^3$. In terms of Theorem \ref{Riesurfcor} these are determined by the parity of the degree of the holomorphic line bundle. We have the trivial bundle $\grS_g\times S^3$ when $n$ is even, and $\grS_g\tilde{\times}S^3$ when $n$ is odd. Combining Theorem \ref{Riesurfcor} with Theorem 4.5 of \cite{BoTo13} gives a complete description of the moduli space of isotopy classes of Sasakian structures on $\grS_g\times S^3$ with $g>1$ whose automorphism group contains a 2-torus. This is given by the $k$-bouquet with $k\in\bbz^+$
$$\gB_k(\cald_k)=\bigcup_{m=0}^{k-1}\bigcup_{\grt\in\gM_g}\bigcup_{\grr\in{\rm Pic}^0(\grS_g)}\grk(\cald_k,J_{\grt,\grr,m})$$
where $c_1(\cald_k)=2-2g-2k$,  and $m=2n\geq 0$. These spaces are non-Hausdorff. A similar bouquet can be given in the genus one case. We mention also that these give all co-oriented contact structures of Sasaki type with a two dimensional maximal torus on $\grS_g\times S^3$ up to contact isotopy.

Another result in this direction is the following.

\begin{thm}\label{theoSPLITTINGtoricSimplex} A toric contact manifold of Reeb type whose moment cone has the combinatorial type of a product of simplices is cone reducible.  In particular, any toric contact structure on an $S^{2k+1}$--bundle over $\bbc\PP^m$ or on a $\bbc\PP^m$--bundle over $S^{2k+1}$ is cone reducible (with $k>0$).   
\end{thm}

\subsection{Does an Extremal Sasaki Metric on a Join Split?}
As in the K\"ahlerian case, extremal Sasakian structures have recently become of much interest, especially their relation to stability properties \cite{CoSz12,BHLT15}. The main result of Section \ref{extsplitsect} is a splitting theorem for extremal Sasakian structures.
It is clear from the join construction that if $M_1$ and $M_2$ both have  extremal Sasaki structures then the standard Sasakian structure on the join will be extremal. The converse of this statement is interesting and not known in full generality. However, we must allow for contact isotopies. Here motivated by the work of \cite{ApHu15,Hua13}, we would like to propose the following question:

\begin{question}
If $M_3$ admits an extremal Sasaki metric by only deforming its contact 1-form, i.e., its Reeb vector field is fixed, then can we conclude that $M_1, M_2$ all admit extremal Sasaki metrics by deforming their contact 1-forms respectively? Alternatively, is a reducible extremal Sasakian structure the join of two extremal Sasakian structures up to contact isotopy?
\end{question}

We give an affirmative answer to the above question when $M_1, M_2$ are toric Sasaki manifolds. 

\begin{thm}\label{SasExtSPLIT}
Let $M_1, M_2$ be two toric quasi-regular Sasaki manifolds and $(l_1, l_2)$ be a pair of relatively prime natural numbers. If the join manifold $M_3 = M_1 \star_{l_1, l_2} M_2$ admits an extremal Sasaki metric by only deforming its contact 1-form, then $M_1, M_2$ also admit extremal Sasaki metrics by deforming their contact 1-forms respectively.
\end{thm}

\begin{ack}
We would like to thank Vestislav Apostolov, David Calderbank, and Tristan Collins for helpful communications and discussions.
Part of the research was done while the authors were visiting the Mathematical Sciences Research Institute (MSRI), Berkeley, CA. They would like to thank MSRI for its hospitality.
\end{ack}

\section{Background}

A Sasaki manifold is an odd dimensional manifold $M^{2n+1}$ endowed with a structure inherited from being a level set in a Riemannian cone with a compatible K\"ahler structure. More precisely, the Riemannian cone manifold $$(C(M), \tilde{g}) = ( M \times \RR^+, r^2g + d r^2)$$ is a K\"ahler manifold with respect to some K\"ahler structure $(\tilde{\omega},\tilde{J})$. A Sasakian structure on $M$ is usually described by a quadruple $\cals=(\xi, \eta, \Phi, g)$, where $\eta = rd^cr$ is a contact 1-form, $\xi =\tilde{J}r\frac{\partial}{\partial r}$ is its Reeb vector field which is a Killing vector field, $\Phi$ is a parallel $(1,1)$ type tensor field determined by $\Phi =- \nabla \xi$, and $g$ is the Riemannian metric that satisfies $g=d\eta\circ(\BOne\otimes \Phi)\oplus \eta\otimes\eta$. The contact form $\eta$ gives rise to the {\it contact structure} which is defined by the codimension one subbundle $\cald$ of the tangent bundle $TM$.

There is a classification of the Sasakian structures by the Reeb foliation $\mathcal{F}_\xi$ on $M = M \times \{1\} \subset C(M)$ generated by the Reeb vector field $\xi$. When $M$ is compact we say that it is a {\it quasi-regular} Sasaki manifold if the orbit of $\xi$ corresponds to a locally free isometric action of the circle group $S^1=U(1)$. If this action is free, then we say that $M$ is a {\it regular} Sasaki manifold. Otherwise, when $M$ is compact and $\xi$ has a non-compact orbit, then $M$ is an {\it irregular} Sasaki manifold.  For further details about Sasaki manifolds we refer to \cite{BG05} and references therein. Unless stated to the contrary we shall assume that our Sasaki manifolds are compact and connected.

\subsection{The Space of Sasakian Structures}\label{subSECTbackground1}
It is well known that a Sasakian structure $\cals=(\xi, \eta, \Phi, g)$ has a transversal K\"ahlerian structure which is determined by $(\cald = \ker \eta, \Phi |_D,  d \eta)$ with a K\"ahler form $d \eta$ and a transverse metric 
$$g^T(\cdot, \cdot) =d\eta(\cdot, \Phi \cdot).$$
A smooth function $f \in C^\infty(M)$ is called {\it basic} if $\xi(f) = 0$. A p-form $\theta$ is called {\it basic} if
$$
i_\xi (\theta) = 0, \quad \pounds_\xi \theta = 0.
$$
Let $C_B^\infty(M)$ be the space of all smooth basic functions on $M$. Then the space of transversal K\"ahler potentials is
$$
\mathcal{H} = \{ \varphi \in C_B^\infty(M) ~|~ d \eta_\varphi = d \eta + \sqrt{-1} \partial \bar{\partial} \varphi > 0 \}.
$$

Now we consider the set
$$
\mathcal{S} (\xi) = \{ \mathrm{Sasakian~ structure}~ (\tilde{\xi}, \tilde{\eta}, \tilde{\Phi}, \tilde{g}) ~|~ \tilde{\xi} = \xi \}.
$$
$\mathcal{S} (\xi)$ is called the {\it space of Sasakian structures compatible with $\xi$}. We are also interested in the set
$$
\mathcal{S} (\calf_\xi) = \{ \mathrm{Sasakian~ structure}~ (\tilde{\xi}, \tilde{\eta}, \tilde{\Phi}, \tilde{g}) ~|~ \tilde{\xi} = a^{-1}\xi,~a\in\bbr^+ \}.
$$

Let $V(\mathcal{F}_\xi)$ be the vector bundle whose fiber at a point $p \in M$ is $T_p M / L_\xi(p)$, where $L_\xi$ is the line bundle generated by $\xi$. Let $\pi: TM \to V(\mathcal{F}_\xi)$ be the natural projection. We can define a complex structure $\bar{J}$ on $V(\mathcal{F}_\xi)$ as follows:
$$
\bar{J} (\pi(X)) := \pi (\Phi(X)),
$$
where $X$ is any vector field in $M$. We define $\mathcal{S}(\xi, \bar{J})$ to be the subset of all Sasakian structures $(\tilde{\xi}, \tilde{\eta}, \tilde{\Phi}, \tilde{g})$ in $\mathcal{S}(\xi)$ with the same complex normal bundle $(V(\mathcal{F}_\xi), \bar{J})$. A choice of Sasakian structure $\cals\in \mathcal{S}(\xi, \bar{J})$ gives a splitting of the tangent bundle $TM=\cald\oplus L_\xi$, an isomorphism $\cald\approx V(\calf_\xi)$, and a strictly pseudoconvex CR structure $(\cald,J)$.

\subsection{Extremal Sasakian Structures}\label{extsassect}
As in K\"ahler geometry, we are particularly interested in certain preferred Sasakian structures called {\it extremal} Sasakian structures which were first defined in \cite{BGS06}. Given a transversal K\"ahler form $d\eta_\varphi$ with $\varphi \in \mathcal{H}$, we denote its scalar curvature by $R_\varphi$.  
Following \cite{BGS06} we define the following energy functional
$$
\begin{array}{cll}
E:\mathcal{S}(\xi, \bar{J}) & \ra{2.0} & \RR, \\
\cals & \mapsto & \int_M R_g^2 ~ d \mu_g,
\end{array}
$$
where $R_g$ is the scalar curvature of $g$ and $\mu_g$ is the volume form. Recall that a vector field $X$ is {\it transversally holomorphic} if $\pi(X)$ is holomorphic with respect to $\bar{J}$. A real basic function $\varphi$ is a real {\it transversally holomorphic potential} if  ${\rm grad}_g\varphi$ is transversally holomorphic.

\begin{Def}
We say that $\cals=(\xi, \eta, \Phi, g) \in \mathcal{S}(\xi, \bar{J})$ is an extremal Sasakian structure if ${\rm grad}_gR_g$ is a real transversally holomorphic vector field.
\end{Def} 
When this happens we also say that the Sasaki metric $g$ is an {\it extremal} Sasaki metric. In \cite{BGS06} it is shown that a Sasakian structure $\cals \in \mathcal{S}(\xi, \bar{J})$ is a critical point of the energy functional $E$ iff the transverse K\"ahler metric $g^T=d\eta\circ (\BOne\oplus \Phi)$ is an extremal K\"ahler metric. Thus we have

\begin{prop}
A transversal K\"ahler form $d\eta_\varphi$ defines a transversally extremal K\"ahler metric if and only if $R_\varphi$ is a real transversally holomorphic potential.
\end{prop}
We emphasize the following obvious statement. A constant scalar curvature (CSC) Sasaki metric is extremal. We also mention the relation $R_g=R_{g^T}-2n$ where the Sasaki manifold $M$ has dimension $2n+1$ and $R_{g^T}$ is the scalar curvature of the transverse K\"ahler metric $g^T$.

\subsection{The Join Construction}
The general join construction was described in \cite{BGO06} (see also Section 7.6.2 of \cite{BG05}).
Let $M_i,  ~ i=1,2$ be two compact quasi-regular Sasaki manifolds with Sasakian structures $(\xi_i, \eta_i, \Phi_i, g_i)$. Let $(l_1, l_2)$ be a pair of relative prime positive integers. Since the Reeb vector field $\xi_i$ generates a locally free circle action, the quotient is a K\"ahler orbifold $\calz_i$ with K\"ahler form $\omega_i$ satisfying $\pi_i^* \omega_i = d \eta_i$, where
$$
\pi_i : M_i \to \calz_i
$$
is the natural projection map. Using a transverse homothety we can always assume that $[\omega_i]$ is a primitive element in $H^2_{orb}(\calz_i, \ZZ)$. Then the product orbifold $\calz = \calz_1 \times \calz_2$ admits a primitive K\"ahler form $[\omega = l_1 \omega_1 + l_2 \omega_2] \in H^2_{orb}(\calz_1 \times \calz_2, \ZZ)$. By the orbifold Boothby-Wang construction \cite{BG00a}, the total space of the $S^1$ V-bundle over $\calz_1 \times \calz_2$ is an orbifold, denoted by $M = M_1 \star_{l_1, l_2} M_2$ with a contact 1-form $\eta$ such that $d \eta = \pi^* \omega$, where $\pi : M \to \calz$ is the natural projection. We have $\eta=\eta_{l_1,l_2} = l_1 \eta_1 + l_2 \eta_2$, and the Reeb vector field $\xi_{l_1,l_2}$ on $M$ is given by
\begin{equation}\label{Reebjoin}
\xi_{l_1,l_2} = \frac{1}{2 l_1} \xi_1 + \frac{1}{2 l_2} \xi_2.
\end{equation}
Note also that $M_1 \star_{l_1, l_2} M_2$ is the quotient orbifold of $M_1\times M_2$ by the circle action generated by the vector field 
\begin{equation}\label{Leqn}
L_{l_1,l_2}=\frac{1}{2 l_1} \xi_1 - \frac{1}{2 l_2} \xi_2.
\end{equation}
The Sasaki orbifold $M_{l_1,l_2}=M_1 \star_{l_1, l_2} M_2$ called the $(l_1,l_2)$-{\it join} of $M_1$ and $M_2$, or just the {\it join} of $M_1$ and $M_2$ when $l_1,l_2$ are understood. Generally, the topology of $M_1 \star_{l_1, l_2} M_2$ depends on $l_1$ and $l_2$. Note that we have the commutative diagram:
\begin{equation}\label{joincomdia}
\begin{matrix}  M_1\times M_2 &&& \\
                          &\searrow\pi_{L_{l_1,l_2}} && \\
                          \decdnar{\pi_{2}} && M_{l_1,l_2}=M_1\star_{l_1,l_2}M_2 &\\
                          &\swarrow\pi_1 && \\
                         \calz_1\times\calz_2 &&& 
\end{matrix}
\end{equation}
where the $\pi$'s are the obvious projections.

The following result is proved in \cite{BGO06}:
Recall that since $M_i$ are quasi-regular Sasaki manifolds, then $\mathcal{F}_{\xi_i}$ is generated by a locally free circle action. Thus the isotropy groups are finite cyclic groups. We denote by $\upsilon_i$ the order of $M_i$, i.e., the lcm of the orders of the isotropy groups.
\begin{prop}\label{smoothprop}
$M_{l_1,l_2}$ is a smooth Sasaki manifold iff $\gcd(l_1 \upsilon_2, l_2 \upsilon_1) = 1$.
\end{prop}

\subsection{The Multiplicative Structure of the Join}
As in the beginning of Section 7.6.2 of \cite{BG05} we denote the set of compact quasi-regular Sasaki manifolds (orbifolds) by $\cals\calm (\cals\calo)$ where $\cals\calo$ is given the $C^{m,\gra}$ topology. $\cals\calm$ is a subspace of $\cals\calo$ and $\cals\calo$ is graded by dimension. Let $\cals\calo_{2n+1}$ denote the subset of Sasakian orbifolds having dimension $2n+1$. We have\begin{equation}\label{gradeddim}
\cals\calo =\bigoplus_{n=0}^\infty \cals\calo_{2n+1}\, ,
\end{equation}
and similarly for $\cals\calm$ manifolds. Note that here we consider $\cals\calo_1=\cals\calm_1$ to consist of one element, namely the circle $S^1$ with its ``Sasakian structure'' $\cals^1=(\partial_t,dt,dt^2)$. Here the transverse structure is a point and the Reeb orbit is the entire Sasaki manifold.

For every pair of relatively prime positive integers $(l_1,l_2)$ the join operation defines a graded multiplication 
\begin{equation}\label{joinmaps}
\star_{l_1,l_2}:\cals\calo_{2n_1+1}\times \cals\calo_{2n_2+1}\ra{1.5} \cals\calo_{2(n_1+n_2)+1}
\end{equation}
which satisfies the ``commutivity relation'' $\calo_1\star_{l_1,l_2}\calo_2=\calo_2\star_{l_2,l_1}\calo_1$, and a partial ``associativity'' relation in the form 
$$(\calo_1\star_{l_1,l_2}\calo_2)\star_{l_3,l_2l_4}\calo_3=\calo_1\star_{l_1l_3,l_2}(\calo_2\star_{l_3,l_4}\calo_3).$$
Notice that $l_2l_4$ and $l_1l_3$ are composite, so this ``associativity relation'' does not hold in general.
By Proposition \ref{smoothprop} the restriction of $\star_{l_1,l_2}$ to $\cals\calm\times \cals\calm$ is in $\cals\calm$ only for those pairs of elements which satisfy $\gcd(l_1\upsilon_2,l_2\upsilon_1)=1$. Note that $\cals^1\star_{l_1,l_2}\cals^1=\cals^1$ by a change of variables, and that $\cals^1\star_{1,1}$ ($\star_{1,1}\cals^1$) is the left (right) identity on $\cals\calo$.

We have a map 
\begin{equation}\label{covers}
\star_{l_1,l_2}\cals^1:\cals\calo_{2n+1}\ra{2.0} \cals\calo_{2n+1}
\end{equation}
defined by sending the Sasakian orbifold $\calo^{2n+1}$ to the join $\calo^{2n+1}\star_{l_1,l_2}\cals^1$.

\begin{prop}\label{sascover}
$\calo^{2n+1}\star_{l_1,l_2}\cals^1$ is the Sasakian structure on $\calo^{2n+1}/\bbz_{l_1}$ with Reeb vector field $l_2\xi$ where $\xi$ is the Reeb vector field for $\calo^{2n+1}$. Hence, the map $\star_{l_1,l_2}\cals^1$ is an $l_1$-fold covering map, and the map $\star_{1,l_2}\cals^1:\calo^{2n+1}\ra{1.6} \calo^{2n+1}$ is a rescaling by a transverse homothety.
\end{prop} 

\begin{proof}
The $S^1$ action on $\calo^{2n+1}\times S^1$ is defined by 
$$e^{i\theta}\cdot(x,e^{it})\mapsto (e^{il_2\theta}\cdot x,e^{i(-l_1\theta +t)}).$$
We do this in stages. First we divide by the $\bbz_{l_1}$ subgroup of $S^1$ to give $\calo^{2n+1}/\bbz_{l_1}\times S^1$. 
Then the residual circle action is given by  
$$e^{i\theta}\cdot (x,e^{it})=  (e^{i\frac{l_2}{l_1}\theta}\cdot x,e^{i(-\theta +t)}).$$
Thus, if $\xi$ is the Reeb vector field for $\calo^{2n+1}$ the induced Reeb vector field on $\calo^{2n+1}\star_{l_1,l_2}\cals^1=\calo^{2n+1}/\bbz_{l_1}$ is $l_2\xi$.
\end{proof}

Of course, there is a similar result for the map 
\begin{equation}\label{covers2}
\cals^1\star_{l_1,l_2}:\cals\calo_{2n+1}\ra{2.0} \cals\calo_{2n+1}
\end{equation}

\begin{example}
Take $\calo^{2n+1}$ to be the sphere $S^{2n+1}$ with its weighted Sasakian structure $\cals_\bfw=(\xi_\bfw,\eta_\bfw,\Phi,g_\bfw)$ which we denote by $S^{2n+1}_\bfw$. Then $\cals^1\star_{l_1,l_2}S^{2n+1}_\bfw$ is the general lens space $L(l_2;l_1w_1,\cdots,l_1w_n)$ and $S^{2n+1}_\bfw\star_{l_1,l_2}\cals^1$ is $L(l_1;l_2w_1,\cdots,l_2w_n)$.
\end{example}

\section{The Join construction and reducibility}\label{joinredsect}
Note that by construction the join $M_1 \star_{l_1, l_2} M_2$ has reducible transverse holonomy. In Definition 7.6.11 of \cite{BG05} any Sasakian structure that can be obtained as the join of Sasakian structures was called reducible. However, to be consistent with the usual terminology involving the de Rham decomposition Theorem we shall refer to this as decomposable and retain the term reducible for the He-Sun definition which occurs at the tangent bundle level.

\begin{Def}\label{redder}
A quasi-regular Sasakian structure $\cals=(\xi,\eta,\Phi,g)$ on an orbifold is {\bf Sasaki decomposable} or just {\bf decomposable} if it can be written as the $(l_1,l_2)$-join of two Sasaki orbifolds of dimension greater than or equal to three; otherwise, it is {\bf indecomposable}. 
\end{Def}
We note that Definition \ref{redder} only applies to quasi-regular Sasakian structures. We want to extend this definition to all Sasakian structures, and relate the irreducible pieces to the multifoliate structures of Kodaira and Spencer \cite{KoSp61}. The main results of this section are: (1) that under a certain technical assumption on the Sasaki automorphism group a reducible Sasakian structure must necessarily be quasi-regular, and (2) a Sasaki version of the de Rham decomposition Theorem holds; hence, a reducible Sasakian structure on a compact simply connected manifold must be a join. Interestingly, the proof of (1) uses the polyhedral structure of the moment cone even in the non-toric case.

\subsection{Reducible Sasakian Structures}
In a recent paper \cite{HeSu12b}  He and Sun have given such an extension of Sasakian reducibility and then have proven that irregular Sasakian structures are irreducible whenever the irreducible pieces have positive Ricci curvature or are flat. In particular, this shows that there can be no join construction for irregular Sasakian structures with positive Ricci curvature or zero curvature on the irreducible subbundles. Their proof also shows that their definition extends ours. We recall the induced connection on the subbundle $\cald$ of $TM$
$$\nabla_X^TY=\begin{cases} (\nabla_XY)^\cald ~\text{if $X$ is  a section of $\cald$}; \\
                                                   [\xi,Y]^\cald~ \text{if $X=\xi$},\end{cases} $$
where $(\cdot)^\cald$ denotes projection onto $\cald$.                                                   

\begin{Def}[He-Sun]\label{HSdef}
A Sasakian structure $\cals=(\xi,\eta,\Phi,g)$ is {\bf reducible} if 
\begin{enumerate}
\item the contact bundle $\cald$ splits as an orthogonal direct sum $\cald=\cald_1\oplus \cald_2$ of non-trivial subbundles; 
\item for $i=1,2$, $\Phi\cald_i=\cald_i$;
\item if $Y$ is a section of $\cald_i$, then so is $\nabla_X^TY$;
\item the transverse metric $g^T$ splits as $g^T=g^T_1+g^T_2$ where $g^T_i=g^T|_{\cald_i}$.
\end{enumerate}
If there is no such splitting of $\cald$, then $\cals$ is called {\bf irreducible}
\end{Def}

Hereafter, by {\it reducible (irreducible)} we shall mean He-Sun reducible (irreducible).
Following \cite{HeSu12b} we define subbundles $\cale_i=\cald_i+L_\xi$ for each $i=1,2$. It is shown that for a reducible Sasakian structure the subbundles $\cale_i$ are integrable, and define foliations on $M$ with totally geodesic leaves. Actually if $\cald$ splits further into subbundles as
$$\cald=\cald_1\oplus \cdots\oplus\cald_r$$
we obtain $r$ integrable subbundles $\cale_i$ for $i=1,\ldots,r$ with $\dim\cald_i=2n_i$ and $\sum_{i=1}^rn_i=n$. This gives rise to a {\it multifoliate structure}. Following \cite{KoSp61} we define the set of vector bundles $P_r=\{TM,\cale_1,\ldots,\cale_r,L_\xi\}$ with a partial ordering given by inclusion. We put $n_0=0$ and then define a surjective map $\psi:\{0,\ldots,2n+1\}\ra{1.5} P_r$ by 
$$\psi(j)=\begin{cases} TM,    &\text{for $j=0$;}  \\
                                       \cale_k     &\text{for $2\sum_{i=0}^{k-1}n_i +1\leq j\leq 2\sum_{i=0}^kn_i$ for $k=1\ldots,r$;} \\
                                       L_\xi        &\text{for $j=2n+1$}.       
                                       \end{cases}          $$
The partial ordering on $P_r$ induces a partial ordering on the set of integers $\{0,\ldots,2n+1\}$ (not the usual one) which we denote by $\gtrsim$.
So $k\gtrsim j$ if and only if $\psi(j)\subset \psi(k)$.
This gives rise to an integrable $G_{P_r}$ structure where $G_{P_r}\subset GL(2n+1,\bbr)$ is the subgroup of all matrices $G$ such that $G^j_k=0$ when $k\not\gtrsim j$ where $j$ labels the columns and $k$ labels the rows. Explicitly this is a $2n+1$ by $2n+1$ matrix with the first $2n$ entries being block diagonal: 
\begin{equation}\label{Gmulfol}
G=\begin{pmatrix}D_1 & 0 & 0  & 0 &*\\
                          0 & D_2 & 0 & 0  & *\\
                          0 &  0 & \ddots & 0 & * \\
                          0 & 0 & 0 & D_r  & *\\
                          0  & 0 &  0 & 0 & *
                          \end{pmatrix}
\end{equation}
where $D_i$ is a $2n_i$ by $2n_i$ matrix. Note that there is a nesting of multifoliate structures, that is, if $\calf_r$ denotes the multifoliate structure defined by the group \eqref{Gmulfol}, then we have nestings $\calf_1\supset \calf_2\supset\cdots\supset\calf_r$ which can occur in many inequivalent ways. Note that the case $r=1$ is just a foliation.
We are particularly interested in the case $r=2$. Of course, when we consider $\calf_2$ it doesn't mean that the subbundles $\cald_1$ and $\cald_2$ are irreducible.                      

We shall refer to the subbundles $\cale_i$ as {\it Sasaki subbundles}. Clearly, we have $\cale_i\cap\cale_j=L_\xi$ for $j\neq i$, and from (2) of Definition \ref{HSdef} we have $\Phi\cale_i=\cald_i$. We define $\Phi_i=\Phi |_{\cale_i}$, $\eta_i=\eta|_{\cale_i}$, and $g_i= g|_{\cale_i\times \cale_i}$. So on the subbundle $\cale_i$ we have the {\it Sasakian structure} $\cals_i=(\xi,\eta_i,\Phi_i,g_i)$. 

Let us now consider the case $r=2$. We have 
\begin{equation}\label{emet}
g_1=g^T_1+\eta_1\otimes \eta_1,\qquad  g_2=g^T_2+\eta_2\otimes \eta_2.
\end{equation}
Here the transverse metric $g^T_i$ means transverse to the characteristic foliation $\calf_\xi$ on the leaves of the foliation $\calf_{\cale_i}$. So it is important to note that the transverse metric to $\cale_1$ is $g^T_2$, a metric on $\cald_2$, and that transverse to $\cale_2$ is $g^T_1$, a metric on $\cald_1$. With this in mind we have the decompositions $TM=\cale_1\oplus \cald_2=\cale_2\oplus\cald_1$, so $\cale_1^\perp=\cald_2$ and $\cale_2^\perp=\cald_1$. 
Now generally leaves of foliations are immersed submanifolds, and  we have

\begin{prop}\label{immsub}
Each leaf $L_i$ of $\cale_i$ is a totally geodesic immersed Sasakian submanifold with the Sasakian structure $\cals_i$.
\end{prop}

\begin{proof}
That the leaves of $\cale_i$ are totally geodesic was proved by He and Sun \cite{HeSu12b}. That the leaves are Sasakian follows from Definition \ref{HSdef} and Okumura's Theorem (\cite{BG05}, Theorem 7.6.2). 
\end{proof}

Proposition \ref{immsub} implies

\begin{cor}\label{efolcomp}
The leaves of the foliation $\cale_i$ are complete with respect to the metric $g_i$.
\end{cor}

As a Sasakian structure has much more than a foliate structure, so then a reducible Sasakian structure has more than a multifoliate structure.
Recall from Proposition 6.4.8 in \cite{BG05} that  the characteristic foliation $\calf_\xi$ of a K-contact (hence a Sasaki manifold) $M^{2n+1}$ is an oriented Riemannian foliation, that is it is a transverse G-structure with the transverse frame group $SO(2n,\bbr)$. For transverse G-structures see Chapter 2 of \cite{Mol88}. But a Sasakian structure also has transverse holomorphic structure, so a Sasaki manifold has a transverse K\"ahlerian structure, that is the transverse G-structure has the transverse frame group $U(n)$. In the case of a reducible Sasakian structure we have

\begin{prop}\label{riemfol}
The foliations $\cale_i$ are K\"ahlerian with respect to the Sasaki metric $g$, that is, $g$ is bundle-like for the foliation $\cale_i$ whose transverse metrics $g^T_{i+1}$ are K\"ahler where $i+1$ is understood to be mod 2. 
\end{prop}

\begin{proof}
It is a well known result (cf. \cite{BG05}, Proposition 2.5.7) that a foliation $\cale$ on $M$ is Riemannian if and only if it admits a bundle-like metric, that is a Riemannian metric $g$ such that if $V$ is a vector field along the leaves of the foliation $\cale$ and $X,Y$ are horizontal foliate vector fields, then $Vg(X,Y)=0$. In our case horizontal foliate with respect to $\cale_1$ corresponds to sections of $\cald_2$ which are independent of the leaf variables of the foliation $\cale_1$, that is, a section $X$ of $\cald_2$ is foliate if $[V,X]\in\Gamma(\cale_1)$ whenever $V\in\Gamma(\cale_1)$. Observe that in that case, condition (3) of Definition \ref{HSdef} implies $$\nabla_VX= \eta(V)\Phi(X).$$ Indeed, $\nabla_VX- \nabla_X V = [V,X] \in \Gamma(\cale_1)$ but the only part of $\nabla_X V$ lying in $\cald_2$ is $\eta(V)\Phi(X)$. Hence, 
\begin{equation}\begin{split}
V\cdot g(X,Y) &= g(\nabla_VX,Y) + g(X,\nabla_VY) \\
&=\eta(V)(g(\Phi(X),Y)+ g(X,\Phi(Y))) \\
&=\eta(V)(g(\Phi(X),Y)+ g(\Phi(Y),X)) \\
&=\eta(V)(d\eta(\Phi(X),\Phi(Y))+ d\eta(\Phi(Y),\Phi(X)))=0.
\end{split}\end{equation}
\end{proof}

However, it follows from Proposition \ref{riemfol} and Definition \ref{HSdef} that more is true, namely
\begin{cor}\label{prodKah}
Let $(M,\cals)$ be a reducible Sasaki manifold with $\dim_\bbr\cald_i=2n_i$. Then $(M,\cals)$ has a transverse K\"ahlerian product structure with transverse frame group $U(n_1)\times U(n_2)$ where $n_1+n_2=n$ and the transverse holonomy group lies in $U(n_1)\times U(n_2)$.
\end{cor}

For a multifoliate Sasakian structure $\calf_r$ the transverse K\"ahler structure to $\cale_i$ is a product of $r-1$ K\"ahler structures with product metric $g^T_{\hat{i}}=g^T_1+\cdots +g^T_{i-1}+g^T_{i+1}+\cdots +g^T_r$, and the transverse holonomy group lies in $U(n_1)\times\cdots\times  U(n_{i-1})\times U(n_{i+1})\times\cdots\times U(n_r)$. 

It is important to realize that Sasakian structures come in rays obtained from a transverse homothety\cite{BG05}. So the induced Sasakian structure $\cals_i$ on the leaves $L_i$ of $\cale_i$ described above gives the ray $\gr_i=\{aS_i\}$ of Sasakian structures defined by the Sasakian structures $aS_i=(a^{-1}\xi,a\eta,\Phi,ag_i+(a^2-a)\eta_i\otimes\eta_i)$ for $a\in\bbr^+$. We now have

\begin{prop}\label{joinred}
The $(l_1,l_2)$-join $M_1\star_{l_1,l_2}M_2$ is He-Sun reducible. Equivalently, a decomposable Sasakian structure is reducible.
\end{prop}

\begin{proof}
As usual we consider vector fields on $M_1\star_{l_1,l_2}M_2$ as vector fields on $M_1\times M_2$ $\mod L_{l_1,l_2}$ where $L_{l_1,l_2}$ is given by Equation \eqref{Leqn}. So by Equation \eqref{Reebjoin} the Reeb vector field on $M_1\star_{l_1,l_2}M_2$ is $\xi=\frac{1}{l_1}\xi_1+\frac{1}{l_2}\xi_2$. Now from the join commutative diagram \eqref{joincomdia} we have $\cald$ splits as $\cald=\cald_1\oplus \cald_2$. Moreover, since the transverse metric also splits as $g^T=g^T_1+g^T_2$, the splitting of $\cald$ is orthogonal. But the transverse complex structure $J=J_1+J_2$ also splits, so one easily sees that $\Phi\cald_i=\cald_i$. Since the transverse connection $\nabla^T$ on $\cald$ is just the Levi-Civita connection lifted from the orbifold $\calz_1\times \calz_2$ item (3) of Definition \ref{HSdef} holds as well. This completes the proof of the proposition.
\end{proof}

\begin{rem}\label{Reebjoinrem}
From the structure of the  Reeb vector field $\xi$ on the join we see that the induced Sasakian structure on the leaves $L_i$ of $\cale_i$ is not $\cals_i$ but the transverse homothety translated Sasakian structure $l_i\cals_i$.
\end{rem}

We have the following converse of Proposition \ref{joinred} under the added hypothesis of simple connectivity. First we state an orbifold version of the de Rham decomposition:
\begin{lemma}\label{orbdeRh}
Let $\calz$ be a compact K\"ahler orbifold of complex dimension $n$ with $\pi_1^{orb}(\calz)=\{id\}$. Suppose that $\calz$ has Riemannian holonomy contained in $U(n_1)\times U(n_2)$ where $n_1+n_2=n$, then $\calz$ is isometric to $\calz_1\times \calz_2$.
\end{lemma}

\begin{proof}
This follows as in the proof of Theorem 6.1 of \cite{KoNo63} by working on local uniformizing neighborhoods.
\end{proof}

\begin{prop}\label{redjoin}
Let $\cals=(\xi,\eta,\Phi,g)$ be a reducible Sasakian structure on a simply connected, compact, connected manifold $M$ such that its Reeb vector field $\xi$ is quasi-regular. Then there exist simply connected compact Sasaki manifolds $M_1,M_2$ and a pair of positive integers $(l_1,l_2)$ such that $M=M_1\star_{l_1,l_2}M_2$ is the join of $M_1$ and $M_2$. Equivalently, a reducible quasi-regular Sasakian structure on a simply connected compact manifold is decomposable.
\end{prop}

\begin{proof}
Since $\cals$ is quasi-regular and $M$ is compact, all leaves of the characteristic foliation $\calf_\xi$ are compact. Thus, by Theorem 7.1.3 of \cite{BG05} $M$ is the total space of an $S^1$ orbibundle over a compact projective algebraic orbifold $\calz$. Furthermore, since $M$ is simply connected $\pi_1^{orb}(\calz)=\{id\}$, and since $\cals$ is reducible, it follows from Corollary \ref{prodKah}  that the transverse holonomy group lies in $U(n_1)\times U(n_2)$ where $n_1+n_2=n$. So by Lemma \ref{orbdeRh} $\calz$ is isometric to a product $\calz_1\times \calz_2$ of projective algebraic orbifolds. Then by orbifold Boothby-Wang (cf. Theorem 7.1.3 of \cite{BG05}) there are primitive forms $\gro_1,\gro_2$ on $\calz_1,\calz_2$, and a pair of positive integers $l_1,l_2$, respectively such that the K\"ahler form on $\calz_1\times \calz_2$ can be written as $l_1\gro_1+l_2\gro_2$ and satisfies 
\begin{equation}\label{prodkahform}
d\eta=\pi^*(l_1\gro_1+l_2\gro_2),
\end{equation}
where $\pi:M\ra{1.6} \calz_1\times \calz_2$ is the $S^1$ orbibundle. Then again by the orbifold Boothby-Wang construction there are manifolds $M_1$ and $M_2$ with Sasakian structures $\cals_1=(\xi_1,\eta_1,\Phi_1,g_1)$ and $\cals_2=(\xi_2,\eta_2,\Phi_2,g_2)$ satisfying $\pi_i:M_i\ra{1.6} \calz_i$ with $d\eta_i=\pi^*\gro_i$. Since the K\"ahler classes $[\gro_i]$ are primitive and $\pi_1^{orb}$ is the identity, the manifolds $M_i$ are simply connected.  But also from Equation \eqref{prodkahform} we have that, up to a gauge transformation, $\eta=l_1\eta_1+l_2\eta_2$. But then the Reeb vector fields are related by Equation \eqref{Reebjoin} and the join construction follows.
\end{proof}

It should be clear that one can iterate this procedure.
\begin{rem}\label{rayjoinred}
In the case of an $S^3_\bfw$-join $M_{l_1,l_2,\bfw}=M\star_{l_1,l_2}S^3_\bfw$ described in \cite{BoTo14a}, we have the following uniqueness result. 
The admissible representatives of all rays in the $\bfw$-cone are `irreducible' except that coming from the join construction, namely the ray determined by $\bfv=\bfw$ (see Lemma 6.4 of \cite{BoTo14a}).
\end{rem}

Next we give examples of reducible Sasakian structures that are indecomposable.

\begin{example}\label{rulsurfex}
We begin by constructing well known ruled surfaces that are locally a product as K\"ahler manifolds, but not so globally. Consider $D\times\bbc\bbp^1$ where $D\subset\bbc$ is the unit disk, with the product K\"ahler form (metric) $(\gro_1, \gro_2)$ which are both Fubini-Study metrics. $\gro_1$ is hyperbolic with constant holomorphic sectional curvature $-1$, $\gro_2$ has holomorphic sectional curvature $+1$. Let $\grS_g$ be a Riemann surface of genus $g>1$ which we can represent as a quotient $D/\pi_1(\grS_g)$ where $\pi_1(\grS_g)$ acts on $D$ by deck transformations. Consider a projective unitary representation $\grr:\pi_1(\grS_g)\ra{1.6} PSU(2)$ and form the quotient 
$$\grS_g\times_\grr\bbc\bbp^1=(D\times\bbc\bbp^1)/(\pi_1(\grS_g),\grr(\pi_1(\grS_g)).$$
This has a local product structure as K\"ahler manifolds which is a global product as K\"ahler manifolds if only if $\grr$ is the identity representation. Nevertheless, the diffeomorphism type of $\grS_g\times_\grr\bbc\bbp^1$ is that of $S^2\times S^2$. In fact there is a two parameter family of integral K\"ahler structures on $\grS_g\times_\grr\bbc\bbp^1$ given by $l_1\gro_1+l_2\gro_2$ with $l_1,l_2\in \bbz^+$. The total space $M_{l_1,l_2,\grr}$ of the principal $S^1$ bundle over $\grS_g\times_\grr\bbc\bbp^1$ whose Euler class is $l_1[\gro_1]+l_2[\gro_2]$ has a natural Sasakian structure $\cals_{l_1,l_2}$ with constant scalar curvature. Now by \cite{GaRu85} the holomorphic tangent bundle to $\grS_g\times_\grr\bbc\bbp^1$ splits as a sum of holomorphic line bundles. This implies that the contact bundle $\cald$ on $M_{l_1,l_2,\grr}$ splits as $\cald=\cald_1\oplus\cald_2$. It is then straightforward to check that the Sasakian structure $\cals_{l_1,l_2}$ is reducible; however, it is decomposable only if $\grr$ 
is the identity representation. 

Note that the representation space, up to equivalence under conjugation by $PSU(2)\approx SO(3)$, is the character variety $\calr(\grS_g)$  which has real dimension $6g-6$. From \cite{NaSe65} we know that the smooth locus of $\calr(\grS_g)$, which is represented by the irreducible unitary representations of $\pi_1(\grS_g)$, is the moduli space of stable rank two holomorphic vector bundles on $\grS_g$; whereas, the singular locus consists of the reducible representations which are realized by the polystable, but not stable, rank two bundles.

\end{example}

Recall \cite{BG05} that a Sasakian structure $\cals$ is said to be of {\it positive (negative) type} if its basic first Chern class\footnote{Here to avoid future ambiguity we write the foliation as $\calf_\cals$ instead of $\calf_\xi$ as done in \cite{BG05}.} $c_1(\calf_\cals)$ can be represented by a positive (negative) definite $(1,1)$ form. It is {\it null} if $c_1(\calf_\cals)=0$ and {\it indefinite} otherwise. There are certain cases where quasi-regularity must hold.

\begin{lemma}\label{leavescomp}
Let $\cals=(\xi,\eta,\Phi,g)$ be a reducible Sasakian structure on a compact manifold $M$. Then $\cals$ is quasi-regular if any of the following conditions hold:
\begin{enumerate}
\item the leaves of both foliations $\cale_1$ and $\cale_2$ are compact;
\item the leaves of one foliation, say $\cale_1$, are compact and $\gA\gu\gt(\cals_1)$ has dimension one;
\item the leaves of one foliation, say $\cale_1$, are compact and $\cals_1$ is of negative or null type;
\item $\cals$ is of positive, negative or null type.
\end{enumerate}
\end{lemma}

\begin{proof}
For item (1) as noted in the proof of Lemma 2.2. of \cite{HeSu12b} the intersection of the leaf $L_1$ of $\cale_1$ and $L_2$ of $\cale_2$ through a point $p\in M$ is precisely the Reeb orbit $\calo_p$ through $p$. So when the leaves are compact, so is $\calo_p$. Thus, $\cals$ is quasi-regular. To prove (2) we simply note that since $\ga\gu\gt(\cals_1)$ is one dimensional and $\cals_1$ has compact leaves, it must be quasi-regular. For (3) we note that if $\cals_1$ is null or of negative type with compact leaves then $\ga\gu\gt(\cals_1)$ is one dimensional, so the result follows from (2). For (4) we first note that if $\cals$ is negative or null, then $\ga\gu\gt(\cals)$ is one dimensional, so it must be quasi-regular. For the positive case we note that $c_1(\calf_\cals)=c_1(\calf_{\cals_1})+c_1(\calf_{\cals_2})$, so $\cals$ is positive if and only if each component $\cals_i$ is positive. The result then follows from \cite{HeSu12b}. 
\end{proof}

We can now ask under what conditions could a reducible irregular Sasakian structure $\cals$ exist on a compact manifold. First $\cals$ must be indefinite.

\subsection{Sasaki Automorphisms}
Here we study the automorphisms group $\gA\gu\gt(\cals)$ of a reducible Sasakian structure $\cals=(\xi,\eta,\Phi,g)$ on $M$. The Lie algebra of $\gA\gu\gt(\cals)$ is denoted $\ga\gu\gt(\cals)$. For the next Lemma, we take the following point of view: $\gt^k$ is the Lie algebra of an abstract compact torus $\bbt_k$ of dimension $k$. That is, $\gt^k$ is simply a $k$--dimensionnal vector space and there is a faithful representation of $\phi : \bbt_k\hookrightarrow \gA\gu\gt(\cals)$ and thus an injective Lie algebra morphism $\phi_* : \gt^k\hookrightarrow \ga\gu\gt(\cals)$. We assume the image of $\phi$ is a maximal torus in $\gA\gu\gt(\cals)$ and the Reeb vector field $\xi$ corresponds to a unique vector, still denoted $\xi$, in $\gt^k$.  

\begin{lemma}\label{lemQUOTsplit} Let $\cals$ be a reducible Sasakian structure. Then there exist two subalgebras $\gg_1,\gg_2 \subset \gt_k/\R\xi$ such that \begin{equation}\label{tksplit0}
\gt_k/\bbr\xi = \gg_1\oplus \gg_2. 
\end{equation} Moreover, we have $\phi_*(p^{-1}(\gg_i))\subset \Gamma(\cale_i)$ where $p : \gt_k \rightarrow \gt_k/\R\xi$ is the quotient map.
\end{lemma}

\begin{proof}
Take a cover of $M$ by open subsets $U_\alpha$ and the submersions $\pi_\alpha : U_\alpha \rightarrow V_\alpha \subset \C^n$. If the Sasaki structure is reducible in the sense of He-Sun then we may assume that we get a product K\"ahler structure $h_\gra^1\oplus h_\gra^2$ on each $V_\alpha=V_\gra^1\times V_\gra^2$ with a local action of $\gt^k/\R\xi$ on it. More precisely, $\pi_\alpha$ is a quotient map with respect to the orbits of Reeb vector field $\xi$. Thus, $\phi_*$ descends as an injective Lie algebra morphism, say $\tilde{\phi}_\alpha$, from the vector space (or trivial Lie algebra) quotient $\gt^k/\R\xi$ with image in $\Gamma(V_\alpha, TV_\alpha)$. Because $\pi_\alpha$ is a Riemannian submersion the image of $\tilde{\phi}_\alpha$ is a space of Killing vector fields on $V_\alpha$, that we will call $\gh_\alpha$.    

The condition (3) of the definition~\ref{HSdef} implies that the decomposition $TV_\alpha= (\pi_\alpha)_*(\cald_1\oplus \cald_2) \simeq (\pi_\alpha)_*\cald_1\oplus (\pi_\alpha)_*\cald_2$ is integrable and closed for the Levi-Civita connection. Hence, any such Killing vector field $K_\alpha\in\gh_\gra$ on $V_\alpha$ can be written uniquely as $K_\alpha=K_\alpha^1+K_\alpha^2$ (with $K_i \in (\pi_\alpha)_*\cald_i$ a Killing vector field for $h_\gra^i$). The condition that $(\pi_\alpha)_*\cald_1$ and $(\pi_\alpha)_*\cald_2$ are closed for the Levi-Civita connection implies that $K$ is Killing if and only if $K^1_\alpha$ and $K^2_\alpha$ are both Killing vector fields. So $\gh_\gra$ splits as $\gh_\gra^1\oplus \gh_\gra^2$ which induces a splitting $$\tilde{\phi}_\alpha^{-1}(\gh_\gra^1\oplus \gh_\gra^2) = \gg_1\oplus \gg_2 = \gt^k/\R\xi.$$ This decomposition does not depend on $V_\alpha$ because the lifts of these vector fields on $U_\alpha\cap U_\beta$ satisfy $X^1_\gra+X^2_\gra=X^1_\grb+X^2_\grb +a_{\gra\grb}\xi$ for some function $a_{\gra\grb}\in C^{\infty}(U_\alpha\cap U_\beta,\bbr)$. This is equivalent to $X^1_\gra+X^2_\gra\equiv X^1_\grb+X^2_\grb\mod \grG_{U_\gra\cap U_\grb}(L_\xi)$. \end{proof}

\begin{cor} There are two Abelian Lie subalgebras $\ga_1,\ga_2 \subset \ga\gu\gt(\cals)$ such that $\ga_\epsilon$ are sections of $\cale_\epsilon$ respectively ($\epsilon =1,2$), $\ga_1\cap\ga_2$ is the $1$--dimensional Lie algebra induced by the Reeb vector field $\xi$ and $\ga_1+\ga_2$ is the Abelian Lie algebra of a maximal torus in $\gA\gu\gt(\cals)$.  
\end{cor}

The rest of this subsection is devoted to a more geometric point of view about the decomposition \eqref{tksplit0} and another way to derive it which emphasizes the multifoliate nature \cite{KoSp61} of reducible Sasakian structures. 

The Sasakian structure $\cals=(\xi,\eta,\Phi,g)$ on $M$ fixes an orthogonal splitting $TM=\cald \oplus L_\xi$. Let $\X(M)$ denote the Lie algebra of vector fields on $M$. We wish to consider certain vector fields on $M$ modulo the sections (perhaps local) of $L_\xi$. In particular we can restrict vector fields to any open subset $U\subset M$ to obtain $\X(U)$. We also consider the algebra of local sections $\grG_U(L_\xi)$. Now any element $X\in\X(M)$ can be written uniquely as $X=X^\cald+\eta(X)\xi$. 

Now on $M$ there is the characteristic foliation $\calf_\xi$ and we let $\gf\go\gl(\calf_\xi)$ denote the Lie algebra of foliate vector fields with respect to $\calf_\xi$ (that is, we recall, $X\in \gf\go\gl(\calf_\xi)$ if and only if $[X,Y] \in \grG(L_\xi)$ as soon as $Y\in \grG(L_\xi)$). Note that $\grG(L_\xi)$ is a subalgebra of $\gf\go\gl(\calf_\xi)$, in fact, by the definition of foliate it is an ideal. So we have a well-defined quotient Lie algebra $\gf\go\gl(\calf_\xi)/\grG(L_\xi)$. For any $X\in \gf\go\gl(\calf_\xi)$ we write $X=X^\cald+\eta(X)\xi$, and it is straightforward to check that both components are foliate. We let $\gf\go\gl(\calf_\xi)_\cald$ denote the foliate vector fields that are sections of $\cald$. So as vector spaces we have the identification $\gf\go\gl(\calf_\xi)_\cald\approx \gf\go\gl(\calf_\xi)/\grG(L_\xi)$. Thus, we can give $\gf\go\gl(\calf_\xi)_\cald$ a Lie algebra structure. If $X,Y\in\gf\go\gl(\calf_\xi)_\cald$ then $[X,Y]\in \gf\go\gl(\calf_\xi)_\cald \mod \grG(L_\xi)$.

Summarizing we have

\begin{lemma}\label{caldalg}
Let $\cals=(\xi,\eta,\Phi,g)$ be a Sasakian structure on $M$ with contact bundle $\cald=\ker\eta$. Then the set of foliate sections $\gf\go\gl(\calf_\xi)_\cald$ of $\cald$ can be given the structure of a Lie algebra.
\end{lemma}

Now we know that $\ga\gu\gt(\cals)$ is a Lie subalgebra of $\gf\go\gl(\calf_\xi)$. Let $\ga\gu\gt(\cals)_\cald$ denote the set of $\cald$ components of elements of $\ga\gu\gt(\cals)$. Note that $\ga\gu\gt(\cals)_\cald \not\subset \ga\gu\gt(\cals)$ in general. From Lemma \ref{caldalg} we have

\begin{lemma}\label{autD}
The set $\ga\gu\gt(\cals)_\cald$ can be given the structure of a Lie algebra isomorphic to $\ga\gu\gt(\cals)/\bbr\xi$. 
\end{lemma}

\begin{proof}
Clearly $\ga\gu\gt(\cals)_\cald$ is a vector space since $(X+Y)^\cald=X^\cald+Y^\cald$.
If $X,Y\in \ga\gu\gt(\cals)$ and we write $X=X^\cald +\eta(X)\xi$ and $Y=Y^\cald +\eta(Y)\xi$ we have
$$[X^\cald,Y^\cald]=[X,Y]-[X,\eta(Y)\xi]+[Y,\eta(X)\xi]+[\eta(X)\xi,\eta(Y)\xi]$$
which since elements of $\ga\gu\gt(\cals)$ are foliate implies $[X^\cald,Y^\cald]\equiv [X,Y]\mod \grG(L_\xi)$. But the only elements of $\grG(L_\xi)$ that are in $\ga\gu\gt(\cals)$ are those in $\bbr\xi$. This implies the result.
\end{proof}

Let $\bbt^k$ be a maximal torus of $\gA\gu\gt(\cals)$, where $k$ is the dimension of $\bbt^k$, and let $\gt_k\subset \ga\gu\gt(\cals)$ denote its Lie algebra of vector fields. Here we note that the Reeb vector field $\xi$ is an element of the Lie algebra $\gt_k$, and that every $X\in \gt_k$ (in fact in $\ga\gu\gt(\cals)$) has a component of the form $\eta(X)\xi$ which, of course, is a section of $L_\xi$. Thus, we have

\begin{lemma}\label{abelD}
There is an Abelian Lie subalgebra $\ga\gb_\cald$ of $\ga\gu\gt(\cals)_\cald$ that is isomorphic to $\gt_k/\bbr\xi$.
\end{lemma}

Let $\cals=(\xi,\eta,\Phi,g)$ be a reducible Sasakian structure on a compact (connected) $2n+1$ dimensional manifold $M$ with an underlying multifoliate structure $\calf_r$. We now consider the automorphism group $\gA\gu\gt(\cals)$ (and its Lie algebra $\ga\gu\gt(\cals)$) of the Sasakian structure $\cals=(\xi,\eta,\Phi,g)$ on the compact manifold $M$. In fact, we work mainly with the Lie algebra $\ga\gu\gt(\cals)$. Now any element of $\gA\gu\gt(\cals)$ must leave the multifoliate structure invariant. Infinitesimally, this means that any element $X\in \ga\gu\gt(\cals)$ is a multifoliate vector field, that is, if $V$ is a vector tangent to the leaves of $\cale_i$ then $[X,V]$ is also tangent to the leaves of $\cale_i$.  In local coordinates $(x^1,\cdots,x^{2n+1})$ a multifoliate vector field $X\in \ga\gu\gt(\cals)$ takes the form (\cite{KoSp61} Definition 3.1)
\begin{equation}\label{mfolvec}
X= \sum_{j=1}^{2n+1}X^j\frac{\partial}{\partial x^j} ~\text{with $\frac{\partial X^j}{\partial x^k}=0$ when $k\not\gtrsim j$}.
\end{equation}
In the case of the multifoliate structure $\calf_2$ the coordinates along the leaves of $\cale_1$ are $(x^1,\cdots,x^{2n_1},x^{2n+1})$ and those along the leaves of $\cale_2$ are $(x^{2n_1+1},\cdots,x^{2n_1+2n_2},x^{2n+1})$. So our multifoliate vector fields satisfy
\begin{equation}\label{mfolvec2}
\frac{\partial X^i}{\partial x^j}=0
\end{equation}
when $2n_1+1\leq i\leq 2n_1+2n_2$ and $j\leq 2n_1$, or $i\leq 2n_1$ and  $2n_1+1\leq j\leq 2n_1+2n_2$, or $i\leq 2n_1+2n_2$ and $j=2n+1$. So any multifoliate section of $\cale_1$ takes the form $X_1=\sum_{i=1}^{2n_1}X^i\partial_{x^i} +X^{2n+1}\partial_{x^{2n+1}}$ where $X^i$ depends only on the variables $(x^1,\cdots,x^{2n_1})$ and $X^{2n+1}$ can depend on all variables. Similarly, any multifoliate section of $\cale_2$ takes the form $X_2=\sum_{i=2n_1+1}^{2n_1+2n_2}X^i\partial_{x^i} +X^{2n+1}\partial_{x^{2n+1}}$ where $X^i$ depends only on the variables $(x^{2n_1+1},\cdots,x^{2n_1+2n_2})$ and again $X^{2n+1}$ can depend on all variables. Consequently, any section of $\cale_1\cap\cale_2$, and hence $\xi$ takes the form $\xi=f\partial_{x^{2n+1}}$ where $f$ can be a function of all variables. We have

\begin{lemma}\label{multiaut}
Let $\cals=(\xi,\eta,\Phi,g)$ be a reducible Sasakian structure and let $X\in \ga\gu\gt(\cals)$. Then on a local multifoliate coordinate chart $U$ we have
\begin{enumerate}
\item $X$ is the sum $X=X_1+X_2$ where for $i=1,2$, $X_i$ is a multifoliate section of $\cale_i$ restricted to $U$;
\item $[\xi,X_i]\equiv 0\mod \grG_U(L_\xi)$;
\item $[X_1,X_2]\equiv 0 \mod \grG_U(L_\xi)$;
\item $[\xi,X_1]=-[\xi,X_2]$.
\end{enumerate}
\end{lemma}

\begin{proof}
The proof of both (1) and (2) can be seen easily from the discussion by writing $X$ and $\xi$ out in local multifoliate coordinates, and (3) follows from the fact that $\xi$ lies in the center of $\ga\gu\gt(\cals)$.

By \eqref{mfolvec2} any such multifoliate vector field can be written as
$$X=\sum_{i=1}^{2n_1}X^i\partial_{x^i}+\sum_{i=2n_1+1}^{2n_1+2n_2}X^i\partial_{x^i} +X^{2n+1}\partial_{x^{2n+1}}$$
 where for $i=1,\cdots,2n_1$, $X^i$ depends only on the variables $(x^1,\cdots,x^{2n_1})$, for $i=2n_1+1,\cdots,2n_1+2n_2$, $X^i$ depends only on the variables $(x^{2n_1+1},\cdots,x^{2n_1+2n_2})$, and $X^{2n+1}$ can depend on all variables. 
\end{proof}

When $\cals$ is reducible any $X\in \ga\gu\gt(\cals)$ is multifoliate, so we have an orthogonal splitting 
\begin{equation}\label{autDsplit}
\ga\gu\gt(\cals)_\cald=\ga\gu\gt(\cals)_{\cald_1}\oplus \ga\gu\gt(\cals)_{\cald_2}.
\end{equation} 
In particular, we have an orthogonal splitting of the maximal Abelian Lie algebra:
\begin{equation}\label{abelsplit}
\ga\gb_\cald =\ga\gb_{\cald_1} \oplus \ga\gb_{\cald_2}.
\end{equation}
So the isomorphism $\ga\gb_\cald\approx \gt_k/\bbr\xi$ gives the existence of Abelian subalgebras $\gg_i$ such that 
\begin{equation}\label{tksplit}
\gt_k/\bbr\xi = \gg_1\oplus \gg_2.
\end{equation}

\subsection{The Sasaki Cone, Moment Cone, and Reducibility}\label{conesect} 
Let $\bbt^k$ be a maximal torus of $\gA\gu\gt(\cals)$, where $k$ is the dimension of $\bbt^k$, and let $\gt_k\subset \ga\gu\gt(\cals)$ denote its Lie algebra. The {\it (unreduced) Sasaki cone} \cite{BGS06} $\gt^+_k$ of $\cals$ is by definition the positive cone in the Lie algebra $\gt$ of a maximal torus in $\gA\gu\gt(\cals)$, i.e.
\begin{equation}\label{sascone}
\gt^+_k=\gt^+_k(\cals)=\{b\in\gt_k~|~\eta(X_b)>0\}
\end{equation}
where $X_b$ is the vector field on $M$ corresponding to the Lie algebra element $b$. This gives rise to a Sasakian structure $\cals_b=(X_b,\eta_b,\Phi_b,g_b)$ where 
$$\eta_b := \frac{1}{\eta(X_b)}\eta$$ is a $\bT^k$--invariant contact form whose Reeb vector field is $X_b$. So one can think of the cone $\gt^+_k$  as parametrizing the set of $\bT^k$--invariant Reeb vector fields with underlying CR structure $(\cald,J)$. It is also convenient to think of the Sasaki cone $\gt^+_k$ as the family of Sasakian structures $\{\cals_b~|~b\in\gt_k^+\}$ with underlying CR structure $(\cald,J)$. Changing the CR structure by changing the complex structure $J$ can give rise to bouquets of Sasaki cones with the same underlying contact structure as described in Section 4.4 of \cite{Boy10a}.

We note that the Sasaki cone $\gt_k^+$ is dual to the interior of the {\it moment cone} $C$ which is defined as the image of the moment map $\Upsilon: \cald^o_+\ra{1.6} \gt_k^*$, that is, $C=\Upsilon(\cald^o_+)\subset \gt_k^*$. Here $\cald^o$ is the annihilator of $\cald$ in $T^*M$ which splits as $\cald^o_+\cup \cald^o_-$. For each $\xi\in\gt^+_k$ there is a unique section $\eta$ of $\cald^o_+$ such that $\xi$ is the Reeb vector field of $\eta$. The equation \cite{BG00b} $\eta(\xi)=1$ describes a hyperplane in $\gt^*_k$. Its intersection with the cone $C$ is a polytope $P$ which is the image of the moment map $\mu_\eta$ associated to $\eta$. That is, we have 
$$P=\mu_\eta(M)=\Upsilon\circ\eta(M).$$

The moment cone $C$ of a contact manifold is {\it rational} with respect to the lattice of circle subgroups $\Lambda \subset \gt_k$~\cite{Ler02}, and in the toric case, that is when $k=n+1$, it is even a {\it good} cone in the sense of Lerman~\cite{Ler02}.
Let us describe this explicitly. We recall the following general definition. 
\begin{Def}\label{Rational_Cone}
A polytope (resp. a cone) in an affine (resp. linear) space is said to be {\bf labelled} if we fix, for each codimension one face, an inward normal vector. It is called {\bf rational} with respect to a given lattice if the inward normals, are lying in that lattice\footnote{To recover the original convention introduced by Lerman and Tolman in the rational case, take $m_k\in \Z$ such that $\frac{1}{m_k}u_k$ is primitive in $\Lambda$ so $(P, m_1,\dots m_d,\Lambda)$ is a rational labelled polytope.}. More specifically, a {\bf rational cone} $C\subset \kt^*$ is a polyhedral cone 
$$\{ x\in \kt^*_k\,|\, \langle x, l_i\rangle\geq 0, \;\; i=1,\dots, d \}$$ 
which is rational with respect to $\Lambda\simeq \Z^{n+1}$, the lattice of circle subgroups in $\kt$. The labels (i.e. inward normals) $l_1,\dots,  l_d \in\kt$ are then chosen (uniquely) by requiring that they are all primitive in $\Lambda$. Moreover, the cone is {\bf good} means that for $I_F\subset \{1,\dots, d\}$ we have $$\Lambda\cap \Span_{\R} \{ l_i\,|\, i\in I_F \} = \Span_{\Z} \{ l_i\,|\, i\in I_F \}$$ whenever $F=C \cap \left( \cap_{i\in I_F} \{x\,|\,\langle x, l_i\rangle =0 \right)$ is a (non-empty) face of $C$.
\end{Def}

Now given a strictly convex rational cone $C$ and a Reeb vector $b\in \gt$, one can consider the {\it labelled polytope} $(P_b, \bu_b)$ defined as follow 
\begin{itemize}
 \item $P_b= C\cap\{x \in \kt^*_k \,|\, \langle x, b\rangle =1\}$;
 \item $\bu_b =( [l_1],\cdots, [l_d] )\in \kt_k/\R b$. 
\end{itemize} This is a labelled polytope in the affine space $\cala_b:= \{x \in \kt^*_k \,|\, \langle x, b\rangle =1\}$ whose dual is naturally identified with $\kt_k/\R b$. We say that $(P_b, \bu_b)$ is a {\it characteristic labelled polytope} of $(C,\Lambda)$.  

Whenever $\R b \cap \Lambda \neq \{0 \}$ and $k>1$, the quotient map $\kt_k \rar \kt_k/\R b$ sends $\Lambda$ to a lattice, say $\Lambda_b$ in $\kt_k/\R b$ and $(P_b, \bu_b)$ is rational with respect to $\Lambda_b$. 

The so-called toric case is when $k=n+1$ and the picture described above fits nicely with the Delzant--Lerman--Tolman correspondence. We can state this concisely in the following summary of results.

\begin{prop}\label{bg00b} \cite{Ler02,BG00b}  Given a compact, connected contact manifold $(M,\cald)$ endowed with the contact action of a torus $\bbt^k$. Assume that there exists $b\in \gt_k$ such that $X_b$ is a Reeb vector field for $(M,\cald)$ and $\eta$ is the corresponding $\bbt^k$--invariant contact form. Then, the Reeb vector field $X_b$ is quasi-regular iff $\R b \cap \Lambda \neq \{0 \}$ and this happens iff $(P_b, \bu_b, \Lambda_b)$ is a rational labelled polytope. In that case, $(P_b, \bu_b, \Lambda_b)$ is the moment rational labelled polytope associated to the Hamiltonian compact $\bT/S^1_b$--space (symplectic orbifold) 
$$(M/S^1_b, (d\eta_b)_D, \bT/S^1_b)$$ 
where $S^1_b =\{\exp(tb)\}$ is the circle induced by $b$ in $\bT$.
\end{prop}
\begin{rem} Proposition~\ref{bg00b} in \cite{BG00b} is proved in the toric case. This extends to the general case $k<n+1$ using the work of Lerman in \cite{Ler02} the condition that there exists a $\bbt^k$--invariant contact form (Reeb type condition) implies that $0$ is not in the image of the moment map on the symplectisation of $(M,\cald)$. Therefore, we can use directly Theorem 1.2 of \cite{Ler02} (the condition on the dimension of $\bbt^k$ is not relevant in our case, i.e if $\dim \bbt^k=1$ the image of $\mu$ is just a point in a line).   
\end{rem}

Now given a labelled polytope $(P, \bu)$ with $P \subset \cala$ and $\bu = \{\vec{n_1},\dots, \vec{n_d}\}$ a set of inward normals one can wonder if  $(P, \bu)$ is a characteristic labelled polytope of a good cone. The following result gives an answer.   

\begin{prop}\label{condCarcGC} \cite{Leg10} A labelled polytope $(P, \bu)$, lying in an affine-linear space $\cala$, is a characteristic labelled polytope of a cone if and only if the set of defining affine functions of $P$ $$l_i(\cdot) =\langle \cdot, \vec{n_i} \rangle -\lambda_i$$ span a lattice, say $\Lambda = \Span_\Z\{l_i \;|\; i=1,\dots, d\}$, in $\mbox{Aff}(\cala,\R)$ for which $$C_P =\{x\in \mbox{Aff}(\cala,\R)^*\,|\, \langle x,l_i \rangle \geq 0 \;\; i=1,\dots, d\}$$ is a good cone.\end{prop}

 \begin{rem}\label{remLABELgoodCON}
\noindent \begin{itemize} 
 \item[(i)] For example when $(P, \bu)$ is a labelled simplex then the condition of Proposition~\ref{condCarcGC} is satisfied because the defining affine functions form a basis of $\mbox{Aff}({\kt^*_k},\R)$. Hence each labelled simplex is characteristic to a good rational cone, which is associated via Lerman's construction~\cite{Ler02a} to a sphere.    
 \item[(ii)] Another way to state the condition that $\Span_\Z\{l_i \;|\; i=1,\dots, d\}$ is a lattice is that there exists $d-k$ linearly independant vectors $\vec{k_1},\dots, \vec{k}_{d-k} \in \Z^d$ in the kernel of the map $$\R^d \ni x\mapsto \sum_{i=1}^d x_i l_i.$$ The condition that  $(P, \bu)$ is a labelled rational polytope is that there exists $d-k+1$ linearly independant vectors $\vec{m_1},\dots, \vec{m}_{d-k+1} \in \Z^d$ in the kernel of the map $$\pi(x) := \sum_{i=1}^d x_i \vec{n}_i$$ where again $x\in\R^d$.
 \item[(iii)] The homothety $r\mapsto \frac{1}{r}b$, for $r>0$ corresponds to the homothety $$\left(rP, \left\{l^r_i:=\langle\vec{n}_i, \cdot \rangle -r\lambda_i\right\}_{i=1}^d\right).$$ One can check (using item (iii) for example) that the condition of Proposition~\ref{condCarcGC} holds either for all or for none of the homothetic labelled polytopes.
 \end{itemize}
 \end{rem}

In view of Remark~\ref{remLABELgoodCON} and Proposition~\ref{condCarcGC}, it is convenient to label polytopes with their defining affine functions instead of the normals (or integers). We now adopt this convention. 

\begin{Def} An affine space $\cala$ is a {\bf product} if there exists two non trivial (i.e of dimension greater than or equal to $1$) affine spaces $\cala_1$ and $\cala_2$, and a bijective affine linear map $\phi : \cala_1\times\cala_2 \rar \cala$. In that case, any pair of non-trivial polytopes $P_i\subset \cala_i$ defines a product polytope $P_1\times P_2 \subset \cala$ in an obvious way. Every polytope in $\cala$ which is affinely equivalent to such polytope is called a {\bf product polytope}.   

More generally, we say that a polytope $P$, lying in a arbitrary affine space $\calb$, is a {\bf product polytope}, if there exists an injective affine linear map $\phi : \cala \rar \calb$ from a product affine space $\cala$ that maps a product polytope to $P$.
\end{Def}

This definition applies equally well to labelled polytopes. We now have

\begin{prop}\label{prodpolyquasi}
 Let $(P_\epsilon,\{l_{\epsilon,i}\}_{i=1}^{d_\epsilon})$, for $\epsilon =1,2$ be labelled polytopes. If the product \begin{equation}\label{prodPOL}
 \left(P_1\times P_2, \{l_{1,i}\}_{i=1,\dots,d_1}\cup \{l_{2,i}\}_{i=1,\dots,d_2}\right)
 \end{equation} is characteristic to a rational cone then $(P_\epsilon,\{l_{\epsilon,i}\}_{i=1}^{d_\epsilon})$, for $\epsilon =1,2$, are both rational labelled polytopes characteristic to rational cones. In particular, \eqref{prodPOL} is rational.
\end{prop}

\begin{proof} We introduce some notation, the polytope $P_\epsilon$ lies in an affine-linear space $\cala_\epsilon \simeq \R^{k_\epsilon-1}$ of dimension $k_{\epsilon}-1$ where $\epsilon =1,2$. So we must have $k_\gre\geq 2$. The product $P_1\times P_2$ lies in $\cala_1\times \cala_2$ which has dimension $k-1$, that is $k_1+k_2 -1$.  Abusing the notation a little, we write $l_{1,i}(\alpha,\beta)= l_{1,i}(\alpha)$ and $l_{2,i}(\alpha,\beta)= l_{2,i}(\beta)$ where $(\alpha,\beta)\in\cala_1\times \cala_2$. That is, there are $2$ injections $$\mbox{Aff}(\cala_\epsilon,\R)\hookrightarrow \mbox{Aff}(\cala_1\times \cala_2,\R)$$ whose images overlap on the subspace of constant functions. The $\Z$-submodules of $\mbox{Aff}(\cala_1\times \cala_2,\R)$, namely $$\Lambda_1=\Span_{\Z}\{l_{1,1},\dots, l_{1,d_1}\}\;\; \mbox{ and }\;\;\Lambda_2=\Span_{\Z}\{l_{2,1},\dots, l_{2,d_2}\},$$ are included in $\Lambda$, the $\Z$--span of the labels of the cone $C_{P_1\times P_2}$. Thanks to Proposition~\ref{condCarcGC} and the 
hypothesis that the product 
labelled polytope \eqref{prodPOL} is characteristic to a rational cone, we know that this span, $\Lambda$,  is a lattice. In particular, $\Lambda_1$ and $\Lambda_2$ have no accumulation point and thus are both lattices in $\mbox{Aff}(\cala_1,\R)$ and $\mbox{Aff}(\cala_2,\R)$ respectively. For dimensional reason their ranks are respectively $k_1$ and $k_2$.  We get that both $(P_1,\{l_{1,i}\}_{i=1}^{d_1})$ and $(P_2,\{l_{2,i}\}_{i=1}^{d_2})$ are labelled polytopes characteristic to rational cones. 

Considering the second point of Remark~\ref{remLABELgoodCON}, the kernel of the map $$ \R^{d_1}\times \R^{d_2} \ni (x,y) \; \mapsto \; \sum_{i=1}^{d_1} x_i l_{1,i} + \sum_{j=1}^{d_2} y_jl_{2,j}$$ intersects $\Z^{d_1}\times \Z^{d_2}$ in a $\Z$-submodule of rank at least $d_1 +d_2 -k_1- k_2+1$. Hence, the fact that $\Lambda_1+\Lambda_2$, $\Lambda_1$ and $\Lambda_2$ are lattices in their respective space of definition implies that there exists $(x,y)\in \Z^{d_1}\times \Z^{d_2}$ such that $\sum_i x_i l_{1,i} + \sum_j y_jl_{2,j}\equiv 0$ but none of the functions $\sum_i x_i l_{1,i}$, $\sum_j y_jl_{2,j}$ vanishes identically. Indeed, we have $\mbox{rank}(\Z^{d_\epsilon}\cap \ker(t\mapsto \sum_i t_il_{\epsilon,i})) = d_\epsilon -k_\epsilon$ for $\epsilon=1,2$.  

But $\sum_i x_i l_{1,i}(\alpha) =- \sum_j y_jl_{2,j}(\beta)$ for all $(\alpha,\beta)\in\cala_1\times \cala_2$ implies that it is a (non vanishing) constant. Therefore, if $\Lambda_1+\Lambda_2$ is a lattice then there exists $x \in \Z^{d_1}$ such that $\sum_i x_i l_{1,i}$ is a non vanishing constant. Which in turns implies that $\sum_i x_i \vec{n}_{1,i} =0$ and that $x$ does not lie in the $\Z$--module of rank $d_1-k_1$ that lies in the kernel of $\R^{d_1} \ni t \mapsto \sum_i t_i l_{1,i}$ which is included in the kernel of $\pi(t)=\sum_{i=1}^{d_1} t_i \vec{n}_{1,i}$.  Hence the rank of $(\ker \pi)\cap \Z^{d_1}$ is $d_1-k_1+1$ and thus $(P_1,\vec{n}_{1,1},\dots,\vec{n}_{1,d_1})$ is rational. The argument is the same for $(P_2,\vec{n}_{2,1},\dots,\vec{n}_{2,d_2})$. 
\end{proof}
 
\begin{prop}\label{propREDimpliesREG} Let $\cals=(\xi,\eta,\Phi,g)$ be a reducible Sasakian structure on a compact connected manifold $M$ such that in the decomposition of Lemma~\ref{lemQUOTsplit} none of the summands are trivial. Then $\xi$ is quasi-regular and if $M$ is simply connected then $\cals$ is the join of quasi-regular compact Sasaki manifolds.  
 \end{prop}
 
 \begin{proof} Denote as before $\bbt^k$, a maximal torus of the automorphisms of $\cals$ and $\gt_k$ its Lie algebra in which lies $\xi$. Let $\mu : M \rightarrow \gt^*_k$ be the usual $\eta$--momentum map defined by $\langle\mu, a\rangle = \eta(X_a)$. We will show that the image of $\mu$, say $P$, is a product polytope. We already know by Proposition~\ref{bg00b} that $P$ is a compact polytope, characteristic to a rational cone.

 Pick $\alpha\in \gt^*_k$ such that $\langle\alpha, \xi\rangle =1$ and put $\mu_o= \mu-\alpha$. Note that $\im \mu_o = P-\alpha$, in particular if $\im \mu_o$ is a product, $P$ is a product. Observe also that $\im \mu_o\subset \xi^0$, the annihilator of $\R\xi$ in $\gt^*_k$. In general if $H$ is a vector subspace of another $E$ and $H^0$ is the annihilator of $H$ in $E^*$ then $(E/H)^*=H^0$. Using this fact, on the decomposition of Lemma~\ref{lemQUOTsplit}, we get the identification 
 \begin{equation}\label{decompQUOT}\xi^0 =( \gt^k /\R\xi)^* = \gg_1^* \oplus\gg_2^*.
 \end{equation}

For $i=1,2$, denote $\gh_i =p^{-1}(\gg_i)$ so that $$\gh_1\cap \gh_2 =\R\xi,\; \gt_k = \gh_1+ \gh_2,$$ $\gg_i = \gh_i/\R\xi$ and $\gg_i^*$ is the annihilator of $\R\xi$ in $\gh^*_i$. In particular, $\gg^*_i\subset \gh^*_i$. Denote the inclusions $\iota_i : \gh_i\rar \gt_k$ and consider the momentum maps $\mu_i \rar \gh_i^*$ of the local action of $\gh_i$ defined by $$\mu_i:=\iota^*_i \circ \mu_o : M \rar \gh_i^*.$$ The image $P_i:=\im \mu_i$ is a compact polytope, say $P_i$, because\footnote{One can also prove that the image of $\mu_i$ is a compact polytope using the classical theory, see \cite{At82,Ler02a}.} it is the image of the compact polytope $P-\alpha$ by a linear map $\iota^*_i$.  Moreover, by construction, the polytope $P_i$ lies in the annihilator of $\R\xi$ in $\gh^*_i$, that is, in $\gg_i^*$. 

Up to an additive constant, $\mu_i$ coincides with the restriction of $\mu$ to $\gh_i$. Indeed, if $a_i\in\gh_i\subset \gt_k$, we have  $$\langle\mu, \iota_i(a_i)\rangle  + \langle\alpha, \iota_i(a_i)\rangle  = \langle\mu_i, a_i\rangle.$$ Using this, and the fact that $\cald_1$ and $\cald_2$ are $g$-orthogonal, we get that $\forall (a,b)\in \gg_1\times\gg_2$, the gradients of the functions $\langle \mu_1, a\rangle$ and $\langle \mu_2, b\rangle$ are $g$--orthogonal. In particular, the image of $(\mu_1,\mu_2) : M\rar \gg_1^* \oplus\gg_2^*$ is a product, namely $P_1\times P_2$.

Now, for any $a\in \gt_k$ we can write $a= a_1+a_2$ for $a_i\in\gh_i$. We have 
\begin{equation}\label{splitofMU}\begin{split}
d\langle \mu_o - (\mu_1,\mu_2), a\rangle &= d\langle\mu_o, a\rangle - d\langle\mu_1, a_1\rangle - d\langle\mu_1,a_2\rangle\\
&= d\eta(X_a,\cdot) - d\eta(X_{a_1},\cdot)- d\eta(X_{a_2},\cdot)\\
&= d\eta(X_{a-a_1-a_2},\cdot)\equiv 0
\end{split}\end{equation}
The decomposition $a= a_1+a_2$ is not unique but the values above are independant of the decomposition since $\langle\mu_o, \xi\rangle =0$, $\langle\mu_1, \xi\rangle =0$, $\langle\mu_2, \xi\rangle =0$ and $d\eta(\xi,\cdot)=0$. From~\eqref{splitofMU}, we get that up to an additive constant $\mu_o= (\mu_1,\mu_2)$, where the equality makes sense thanks to the decomposition~\eqref{decompQUOT}. It follows that $\im \mu_o$ is a product polytope and thus $\im \mu = \im \mu_o +\alpha$ is a product polytope characteristic to a rational cone. By Proposition~\ref{prodpolyquasi} both polytopes $P_i$ are rational; hence, the product is rational as well and, using Proposition~\ref{bg00b}, we get that $\xi$ is quasi-regular. Thus, when $M$ is simply connected Proposition \ref{redjoin} implies that $M$ is decomposable. 
\end{proof}

Theorem \ref{theoREDimpliesREG} is an immediate consequence of Proposition \ref{propREDimpliesREG}.

\subsection{Molino Theory and Sasaki Geometry}
In this subsection we briefly review some important invariants of Riemannian foliations due to Molino \cite{Mol79,Mol81,Mol82,Mol84,Mol88} and apply them to Sasakian structures. Recall that a transverse oriented Riemannian structure is a reduction of the transverse frame group to the special orthogonal group. It is well known that the characteristic foliation $\calf_\xi$ of a compact connected Sasaki manifold $M$ is a one dimensional Riemannian foliation, that is, a Riemannian flow. Actually in this case the transverse geometry is K\"ahler. Moreover, as seen from Proposition \ref{riemfol} the foliations $\cale_i$ are also transversely K\"ahlerian, hence, transversely Riemannian. Molino describes two invariants associated to a Riemannian foliation $(M,F)$. The first is called the {\it structural Lie algebra} \cite{Mol81} and denoted $\gg(M,F)$. Its definition requires a transverse parallelism. It is well known that frame bundles are parallelizable, so we can lift a Riemannian foliation $F$ to a foliation $F^1$ on the transverse orthonormal frame bundle\footnote{In the case that $(M,F)$ has a transverse orientation (which happens in our case), we need to choose a connected component of $E^1_T$.} $E^1_T$ of $(M,F)$. Now the closure of the leaves of the lifted foliation $(E^1_T,F^1)$ are the fibers of a locally trivial fibration $\pi^1_T:E^1_T\ra{1.5} W^1_T$, called the {\it basic fibration}, such that $F^1$ induces a foliation on the fibers $N\approx(\pi^1_T)^{-1}(w)$ with $w\in W^1_T$ whose leaves are spanned by a finite dimensional Lie algebra $\gg(E^1_T,F^1)$. Note that the leaves of this induced foliation are dense in $(\pi^1_T)^{-1}(w)$ and the only basic functions of this foliation are the constants. Moreover, Molino shows that $\gg(E^1_T,F^1)$ is independent of the transverse Riemannian structure, so it only depends on the original foliation $(M,F)$. Hence, we define $\gg(M,F)=\gg(E^1_T,F^1)$ and note that it depends only on the foliation $(M,F)$ independent of the transverse Riemannian metric.
 Applying this to the characteristic foliation $\calf_\xi$ of a Sasakian structure, we see that it only depends on the space $\mathcal{S}(\calf_\xi)$ of Sasakian structures, not on a particular Sasakian structure $\cals\in\mathcal{S}(\calf_\xi)$. In particular, $\gg(M,\calf_\xi)$ is independent of the transverse Riemannian metric and the transverse holomorphic structure. Now the closure $\overline{\calf_\xi}$ of leaves of $\calf_\xi$ is a singular Riemannian foliation whose leaves are tori $T^k$ with $1\leq k\leq n+1$ where the dimension of $M$ is $2n+1$. For the rest of this section $T^k$ is the torus generated by the Reeb vector field, and not necessarily a maximal torus in $\gA\gu\gt(\cals)$ as previously.

\begin{lemma}\label{xiLie}
The structural Lie algebra $\gg(M,\calf_\xi)$ of the characteristic foliation $\calf_\xi$ of a Sasaki manifold $M$ is an Abelian Lie algebra of dimension $k-1$ where $k$ is the dimension of the closure of a generic Reeb orbit.
\end{lemma}

\begin{proof}
As discussed by Molino and mentioned above, by passing to a component of the transverse orthonormal frame bundle of $\calf_\xi$ we can assume the foliation is transversely parallelizable. Then by Lemma 4.2 of \cite{Mol88} the induced foliation $(N,\calf_\xi)$ is transversely parallelizable. Now the fibers $N$ of the basic fibration are the leaves of $\overline{\calf_\xi}$, that is, the tori $T^k$, and since a Reeb orbit is dense in $T^k$ the structure Lie algebra $\gg(M,\calf_\xi)$, as Molino shows, is precisely the transverse foliate vector fields which is, in this case, Abelian of dimension $k-1$.
\end{proof}

\begin{rem}\label{Molrem}
Molino describes another invariant which is called the `faisceau transverse central' in \cite{Mol79,Mol82} and the `commuting sheaf' in \cite{Mol88} and denoted by $C(M,F)$. It is a locally constant sheaf of local transverse Killing vector fields which is universal in the sense that it is independent of the transverse Riemannian metric, depending only on the foliation $F$. However, in the case of the characteristic foliation $\calf_\xi$ of a Sasaki manifold $M$, the leaves are geodesics, and it follows from \cite{MoSe85} that $C(M,\calf_\xi)$ is the constant sheaf. In this case elements of $C(M,\calf_\xi)$ are central elements of the Lie algebra of transverse foliate vector fields on $(M,\calf_\xi)$. So, as Molino shows, $C(M,\calf_\xi)$ actually coincides with the structural Lie algebra $\gg(M,\calf_\xi)$. Thus, we can represent elements of $\gg(M,\calf_\xi)$ by transverse Killing vector fields.
\end{rem}

For a reducible Sasakian structure we consider the lattice of vector bundles $P_2=(TM,\cale_1,\cale_2,L_\xi)$ with multifoliate structure $\calf_2$. We are interested in invariants of this multifoliate structure. First we consider the Lie algebras $\gg(M,\cale_i)$ which are invariants of the foliations $\calf_{\cale_i}$. We also have the locally constant Lie algebra sheaves $C(M,\cale_i)$. When $C(M,\cale_i)$ is a constant sheaf it is Abelian and coincides with the structural Lie algebra $\gg(M,\cale_i)$. We have

\begin{prop}\label{multifolinv}
Let $\cals=(\xi,\eta,\Phi,g)$ be a reducible Sasakian structure on the compact manifold $M$ with multifoliate structure $\calf_2$ defined by the subbundles $(\cale_1,\cale_2,L_\xi)$ of $TM$. Then the triple of Abelian Lie algebras $\bigl(\gg(M,\cale_1), \gg(M,\cale_2),\gg(M,\calf_\xi)\bigr)$ satisfying 
$$\gg(M,\calf_\xi)=\gg(M,\cale_1)\oplus \gg(M,\cale_2)$$
is an invariant depending only on the multifoliate structure $\calf_2$. Moreover, $\gg(M,\cale_i)\subset \gg_{i+1}$ where $\gg_i$ are the Abelian Lie algebras of Equation \eqref{tksplit} and ${i+1}$ is understood to be taken mod~2.
\end{prop}

\begin{proof}
The invariants of the multifoliate structure $\calf_2$ are the invariants of the three foliations $\calf_\xi,\calf_{\cale_1},\calf_{\cale_2}$ together with any relations among them. We compute these structural Lie algebras $\gg(M,\cale_i)$ by considering their lifts to $E^1_T$ of the singular foliations $\overline{\calf_{\cale_i}}$ on $M$ defined by the leaf closures of $\cale_i$. We denote the lifted foliation to $E^1_T$ by $\calf^1_{\cale_i}$. The leaves of $\calf^1_{\cale_i}$ are the fibers $N$ of the basic fibration and the foliation on $N$ induced by $\calf^1_{\cale_i}$ is spanned by the structural Lie algebra $\gg(M,\cale_i)$. Now following Molino we know (cf. Theorem 5.2 in \cite{Mol88} and its proof) that the closures of the leaves of $\calf_{\cale_i}$ are the orbits of the locally constant sheaf $C(M,\cale_i)$. The subsheaf of constant elements of $C(M,\cale_i)$  coincide with $\gg(M,\cale_i)$ and consists of transversal Killing vector fields which lie in the center of the Lie algebra of transverse 
foliate vector fields. It follows that $\gg(M,\cale_i)$ are Abelian Lie algebras. 

Now the transverse metric of $\cale_1$ is $g^T_2$ that of $\cale_2$ is $g^T_1$. Consider $C(M,\cale_1)$ and recall that the elements of $C(M,\cale_1)$ are the transverse Killing fields of every transverse metric $g^T_2$ of $\cale_1$. By reducibility we have $g^T=g^T_1\oplus g^T_2$ where $g^T$ is the transverse metric for $\calf_\xi$, and there exist local multifoliate coordinates $(x^1,\cdots,x^{2n_1},x^{2n_1+1},\cdots,x^{2n_1+2n_2},x^{2n+1})$ such that $g^T_1$ only depends on the coordinates $(x^1,\cdots,x^{2n_1})$ and $g^T_2$ on the coordinates $(x^{2n_1+1},\cdots,x^{2n_1+2n_2})$. Moreover, a Killing vector field that is transverse to the leaves of $\cale_1$ takes the form $\sum_{j=2n_1+1}^{2n_1+2n_2}X^j\partial_{x^j}$, that is, it is a section of $\cald_2$ and so if it is a Killing vector field for $g^T_2$ it is also a Killing vector field for $g^T$. Similarly, a transverse Killing vector field for $g^T_1$ is a section of $\cald_1$ and is a Killing vector field for $g^T$ as well. This shows that $\gg(M,\cale_i)\subset \gg(M,\calf_\xi)$ for $i=1,2$. Furthermore, any transverse Killing vector field for the transverse metric $g^T$ of the Sasakian structure $\cals$ splits according to Equation \eqref{autDsplit}. Since the Lie algebra $\gg(M,\calf_\xi)$ is Abelian it must lie in a maximal Abelian Lie algebra of $\ga\gu\gt_\cald(\cals)$ which splits according to Equation \eqref{tksplit}. Furthermore, since the elements of $\gg(M,\cale_i)$ are sections of $\cald_{i+1}$  we have $\gg(M,\cale_i)\subset \gg_{i+1}$. This proves the splitting and the proposition.
\end{proof}

\begin{cor}\label{auts1e2comp}
Let $\cals=(\xi,\eta,\Phi,g)$ be a reducible Sasakian structure on a compact manifold $M$ with foliations $\cale_1$ and $\cale_2$. Suppose further that $\ga\gu\gt(\cals_1)$ is one dimensional. Then the leaves of $\cale_2$ are all compact.
\end{cor}

\begin{proof}
Since $\ga\gu\gt(\cals_1)$ is one dimensional, the subalgebra $\gg_1$ of Equation \eqref{tksplit} vanishes which by Proposition \ref{multifolinv} implies that $\gg(M,\cale_2)=0$. But this implies that the leaves of $\cale_2$ are compact (cf. Proposition 5.4 of \cite{Mol88}).
\end{proof}

Corollary \ref{auts1e2comp} allows us to handle the following cases:
\begin{prop}\label{poss1}
Let $\cals=(\xi,\eta,\Phi,g)$ be a reducible Sasakian structure on a compact manifold $M$ and suppose that $\ga\gu\gt(\cals_1)$ is one dimensional. Then $\cals$ is quasi-regular if either of the following two cases hold
\begin{enumerate}
\item The Sasakian structure $\cals_1$ has positive transverse Ricci curvature.
\item The transverse structure of $\cals_1$ is flat.
\end{enumerate}
\end{prop}

\begin{proof}
If the metric $g^T_1$ that is transverse to the characteristic foliation $\calf_\xi$ of $\cals_1$ has positive Ricci curvature, the leaves of $\cale_1$ are compact by Myers Theorem as pointed out in \cite{HeSu12b}. But by Corollary \ref{auts1e2comp} the leaves of $\cale_2$ are also compact, so the result follows from Lemma \ref{leavescomp}. This proves case (1).

For (2) suppose that the transverse K\"ahlerian structure to $\cals_1$ is flat and that there is a noncompact leaf $\call$ of $\cale_1$. Then the argument in \cite{HeSu12b} shows that $\call/\calf_\xi$ is Hausdorff and a quotient of $\bbc^{n_1}$ by a K\"ahler isometry. By Lemma 2.3 of \cite{HeSu12b} $\call/\calf_\xi$ is noncompact. Then as the argument in the proof of Theorem 2.1 in \cite{HeSu12b} the automorphism group of $\call/\calf_\xi$ contains an element of the form $z\mapsto z+c$ in its noncompact factor which is Hamiltonian. But by Corollary 8.1.9 of \cite{BG05} this would then lift to an element of $\ga\gu\gt(\cals_1)$ contradicting the one dimensionality of $\ga\gu\gt(\cals_1)$.
\end{proof}

\section{Cone reducibility}\label{coneredsect}

In this section we generalize the notion of reducibility to that where there is some Sasakian structure in the Sasaki cone that is reducible.

\begin{Def}\label{conered}
A Sasakian structure $\cals$ is called {\bf cone reducible (decomposable)} if there is a reducible (decomposable) Sasakian structure in its Sasaki cone $\gt^+(\cals)$. If there is no reducible (decomposable) element in $\gt^+(\cals)$ it is called {\bf cone irreducible (indecomposable)}. We also say that the underlying CR structure $(\cald,J)$ is {\bf cone reducible (decomposable)}, etc.
\end{Def}

By Proposition \ref{joinred} a cone decomposable Sasakian structure is cone reducible, and a cone irreducible Sasakian structure is cone indecomposable. Reducibility is a property of a ray of Sasakian structures, whereas, cone reducibility is a property of the family of Sasakian structures belonging to the same Sasaki cone. Of course, a reducible Sasakian structure is cone reducible and if a Sasaki cone is one dimensional, cone reducibility coincides with reducibility of the ray.  An example of a cone irreducible Sasakian structure with a large automorphism group, hence a large Sasaki cone, is the standard Sasakian structure on the sphere $S^{2n+1}$ which clearly is a toric contact structure of Reeb type.  

There are some easy topological consequences of cone reducibility. Since the second Betti number of any compact quasi-regular Sasaki manifold $M$ is one less than the second Betti number of its base orbifold we have

\begin{prop}\label{topconered}
If a compact manifold $M$ admits a cone decomposable Sasakian structure, then $b_2(M)\geq 1$.
\end{prop}

Dimension five is of particular interest. First, the following result which is a special case of Theorem \ref{theoSPLITTINGtoricSimplex} also follows from Lemmas 2.2 and 2.5 of \cite{BoPa10}:   

\begin{prop}\label{s2s3conered}
Every toric contact structure on an $S^3$-bundle over $S^2$ is cone decomposable. 
\end{prop}

Furthermore, since the orbifold base of the reducible ray is a product of weighted projective $\bbc\bbp^1$s, the standard Sasakian structure on the join is extremal. In the next section we show that for toric contact structures of Reeb type the converse is also true. Here is an interesting example:

\begin{example}\label{coneredex}
In \cite{GMSW04a} the physicists introduced a sequence of Sasaki-Einstein manifolds $Y^{p,q}$ depending on a pair of relatively prime positive integers $p,q$ satisfying $1\leq q<p$ which are diffeomorphic to $S^2\times S^3$. Moreover, these structures are toric and their geometry was studied further in \cite{MaSp06}. The Sasaki cone for the case $Y^{2,1}$ has a regular Reeb vector field that fibres over $\bbc\bbp^2\#\overline{\bbc\bbp^2}$, that is $\bbc\bbp^2$ blown-up at a point. This is an irreducible Sasakian structure. However, as seen in \cite{BoPa10,BoTo14a} it is an element of the Sasaki cone of the join $S^3\star_{1,2}S^3_{3,1}$ where here $S^3_\bfw$ is the weighted 3-sphere \cite{BG05} with $\bfw=(3,1)$. Thus, $Y^{2,1}$ is cone decomposable. This is the only $Y^{p,q}$ with a regular Reeb vector field in its Sasaki cone.  
\end{example}

Low dimensions put a constraint on decomposability. For example, in dimension five we have
\begin{prop}\label{irred5man}
Let $M$ be a 5-dimensional compact simply connected Sasaki manifold with $H^2(M,\bbq)\geq 2$. Then $M$ is necessarily cone indecomposable.
\end{prop}

\begin{proof}
Suppose to the contrary that there is a Sasakian structure in the Sasaki cone of $M$ that is decomposable. Then it is necessarily quasi-regular and an $S^1$ orbibundle over a product of one dimensional algebraic orbifolds $\calo_1\times \calo_2$. But since $M$ is simply connected, $\pi_1^{orb}(\calo_1\times \calo_2)=\{id\}$. So by Kunneth $H^2(\calo_1\times \calo_2,\bbq)\approx \bbq^2$ which implies $H_2(M^5,\bbq)\approx \bbq$ which is a contradiction.
\end{proof}

\subsection{$S^3$-bundles over Compact Hodge Manifolds}
We now want to consider contact manifolds $M$ that are $S^3$-bundles over a compact smooth projective algebraic variety $N$. A choice of integral K\"ahler form $\gro_N$ on $N$ is then called a {\it Hodge manifold}. Let $\cald$ be a contact structure on $M$ and let $\gC\go\gn(M,\cald)$ denote the group of contactomorphisms. If $G$ is a Lie subgroup of $\gC\go\gn(M,\cald)$, we say \cite{BG05} that an action $\cala:G\ra{1.6} \gC\go\gn(M,\cald)$ of $G$ is of {\it Reeb type} if there is a contact 1-form $\eta$ such that $\cald=\ker\eta$ and an element $\varsigma$ in the Lie algebra $\gg$ of $G$ such that the corresponding vector field $X^\varsigma$ is the Reeb vector field of $\eta$. Hereafter, we often identify such a vector field with the corresponding element of the Lie algebra. We begin the proof of Theorem \ref{s3riemcor} with several lemmas.

Now as usual we let $\calo (\calo^*)$ denote the sheaf of germs of holomorphic functions (no where vanishing holomorphic functions) on $N$, then the short exact exponential sequence gives the exact cohomology sequence
$$0\ra{2.0} H^1(N,\calo^*)/H^1(N,\calo)\fract{c_1}{\ra{2.0}} H^2(N,\bbz)\ra{2.0}H^2(N,\calo)\ra{2.0}\cdots.$$
The group $H^1(N,\calo^*)$ is also written as ${\rm Pic}(N)$ and is called the {\it Picard group} of holomorphic line bundles on $N$.
The image of $c_1$ in $H^2(N,\bbz)$ is called the {\it Neron-Severi group} $NS(N)$ and its rank is called the {\it Picard number} $\grr(N)$ of $N$. It follows from the Lefschetz theorem on $(1,1)$ classes that any integral K\"ahler class lies in $NS(N)$. Thus, for any integral K\"ahler form $\gro_N$ there exists a complex line bundle $L_1\in {\rm Pic}(N)$ such that $c_1(L_1)=[\gro_N]$. The kernel of $c_1$ is the Picard variety ${\rm Pic}^0(N)$ which is a complex torus of real dimension $b_1(N)$.

There is another complex torus associated to $N$, namely the {\it Albanese variety} of $N$ defined by
$$A(N)=H^0(N,\grO^1)^*/H_1(N,\bbz).$$ 
It is the dual torus to ${\rm Pic}^0(N)$. Moreover, there is a holomorphic map $A(N)\ra{1.6} N$ which induces an isomorphism 
$$H^0(A(N),\grO^1_{A(N)})\approx H^0(N,\grO^1_N).$$

\begin{lemma}\label{Alblem}
If $N=N'\times A(N)$ where $N'$ is a compact connected Hodge manifold with finite automorphism group and $M$ is an $S^3$-bundle over $N$ with an effective $\bbt^2$ action of Reeb type, then $\bbt^2$ acts trivially on $N$.
\end{lemma}

\begin{proof}
Assume the action of $\bbt^2$ on $N$ is non-trivial, then it is non-trivial on $A(N)$ which is a torus. But by Theorem 9.3 in chapter IV of \cite{Bre72} the only effective action of a Lie group on a torus is a free action of a torus. Thus, there are two cases to consider: 
\begin{enumerate}
\item $\bbt^2$ acts freely on $A(N)$, or 
\item an $\bbs^1$ quotient acts freely on $A(N)$. 
\end{enumerate}
In either case since the action is of Reeb type, by Proposition 8.4.30 of \cite{BG05} there is a $\bbt^2$ invariant K-contact structure $\cals=(\xi,\eta,\Phi,g)$ which must be Sasakian since as in the case discussed above the tranverse almost couplex structure is integrable. Moreover, by Theorem 7.1.10 of \cite{BG05} we can take $\cals$ to be quasi-regular, so that the quotient generated by the Reeb vector field is a compact projective algebraic orbifold $S$. Now from the fibration $S^3\ra{1.5}M\ra{1.5} N'\times A(N)$ and the fact that $A(N)=T^{2k}$ is an Abelian variety, we see that $\pi_1(M)\approx \pi_1(N')\times \bbz^{2k}$ which contains $\bbz^{2k}$ as a direct summand. But then since $S$ is projective $\pi_1^{orb}(S)$ must contain $\bbz^{2k}$ as a direct summand. This implies that $H^1(S,\bbr)$ also contains $\bbr^{2k}$ as a direct summand.

Now consider case (1). By a change of coordinates if necessary we can take the $\bbt^2$ action on $A(N)=T^{2k}$ to be translation in the first two coordinates of $T^{2k}$. The corresponding vector field $\check{X}$ is not Hamiltonian with respect to the K\"ahler form $\gro_N$ on $N$. So by Corollary 8.1.9 of \cite{BG05} it does not lift to an infinitesimal automorphism in $\ga\gu\gt(\cals)$. So it cannot give rise to an element in the Lie algebra $\gt_2$ of $\bbt^2$, which is a contradiction.

In case (2) there are two sub cases to consider whether or not the infinitesimal generator $\check{X}$ of the $\bbs^1$-action is a projection of the Reeb vector field or not. If $\check{X}$ is not a projection of the Reeb field, then a similar argument as in case (1) leads to a contradiction. On the other hand if $\check{X}$ were the projection of the Reeb vector field $\xi$, then since this action is free on $A(N)$, the map $\psi$ in the exact sequence
$$\ra{2.0}\bbz\fract{\psi}{\ra{2.0}}\pi_1(N')\times\bbz^{2k}\ra{2.0}S\ra{2.0}\{id\}$$
would inject into $\bbz^{2k}$ in which case $S$ would not be K\"ahler giving a contradiction.
\end{proof}

Next we have

\begin{lemma}\label{contactsub}
Let $M$ be an $S^3$-bundle over a smooth compact algebraic variety $N$, and let $\cald$ be a co-oriented contact structure on $M$ with an effective $\bbt^2$ action of Reeb type that acts trivially on $N$. Then 
\begin{enumerate}
\item each fiber $F_x=S^3$ is a contact submanifold.
\item The contact manifold $(M,\cald)$ is of Sasaki type and any quasi-regular Reeb vector field associated with the $\bbt^2$ action has an orbifold quotient of the form $(S_L,\grD)$ where $S_L=\bbp(\BOne\oplus L)$ where $L$ is a holomorphic line bundle on $N$ and $\grD$ is a branch divisor (possibly empty).
\item $(S_L,\grD)$ admits a holomorphic Hamiltonian $S^1$ action on its fibers.
\end{enumerate}
\end{lemma}

\begin{proof}
Since the action of $\bbt^2$ is of Reeb type and $\cald$ is co-oriented, there is a contact 1-form $\eta$ such that $\cald=\ker\eta$ and whose Reeb vector field $\xi$ lies in the Lie algebra $\gt_2$ of $\bbt^2$. Moreover, since $\bbt^2$ acts non-trivially only on the fiber $F_x\approx S^3$, by identifying elements of $\gt_2$ with the vector fields on $M$ induced by the $\bbt^2$ action, we have a basis $\{\xi,X\}$ for $\gt_2$ that restricted to a fiber $F_x$ is tangent at each point to $F_x$. By Proposition 8.4.30 of \cite{BG05} we can assume that the contact structure is a K-contact structure $(\xi,\eta,\Phi,g)$ and that the $\bbt^2$ action leaves both the endomorphism $\Phi$ and the metric $g$ invariant. But then the vector $\Phi X_x$ belongs to a transversely holomorphic section of $\cald |_F$ and satisfies $d\eta(\Phi X_x,X_x)=g(X_x,X_x)\neq 0$ everywhere. Thus, the restriction $\eta_F$ of the 1-form $\eta$  to $F_x$ satisfies $\eta_F\wedge d\eta_F\neq 0$ for all points of $F_x$. Thus, $\eta_F$ defines a 
contact structure on $F_x$ for every $x\in N$. This proves (1).

Since $\bbt^2$ acts non-trivially only on the fibers, by (1) the contact structure $\cald$ restricts to a contact structure with a $\bbt^2$ action of Reeb type on each fiber $F_x\approx S^3$ which is toric. By a theorem of Eliashberg \cite{Eli92} any contact structure on $S^3$ is either overtwisted or tight and there is a unique, up to oriented isotopy, tight contact structure on $S^3$, namely the standard contact structure $\cald_{st}$ and only the latter is of Reeb type \cite{Ler02a,Boy10a}. Hence, the contact structure on $S^3$ is $\cald_{st}$ and $S^3\ra{1.5} M\fract{\pi}{\ra{1.5}} N$ is a contact fiber bundle \cite{Ler04b} which has a natural fat connection $\calh$ whose curvature is $d\eta |_{\calh\times\calh}$. Now denoting the vertical bundle of the fibration by $\calv$ we have the following decompositions of vector bundles on $M$:
\begin{equation}\label{tmsplit}
TM=\calh+\calv, \qquad \calv=\cald_{st}+L_\xi,\qquad \cald=\calh+\cald_{st}.
\end{equation}
Now since the set of compatible almost complex structures on $\cald$ is contractible, we can choose the transverse almost complex structure to be an integrable transverse complex structure consisting of the complex structure on $N$ lifted to $\calh$ then twisted with the standard $\bbt^2$ invariant transverse complex structure on each fiber $F_x=S^3$. Hence, we can take the K-contact structure $\cals=(\xi,\eta,\Phi,g)$ to be Sasakian, and since $\bbt^2\subset \gA\gu\gt(\cals)$, there is a two-dimensional subcone $\gt^+_2$ of the unreduced Sasaki cone of $M$ giving a two-dimensional family of Sasakian structures on $M$. Since the fibers $F_x$ are contact manifolds the structure group of the $S^3$ bundle reduces to $S^1\subset SO(3)$, so the $\bbt^2$ action commutes with the transition functions of the $S^3$ bundle. Furthermore, since the structure group is linear, the sphere bundle $S^3\ra{1.5} M\fract{\pi}{\ra{1.5}}N$ extends to a rank two holomorphic vector bundle $E$ over $N$ and the $\bbt^2$ action on $M$ extends to a complex linear action on the rank two holomorphic vector bundle $E$. Choosing a quasi-regular Reeb vector field $\xi$ in the Sasaki cone $\gt^+_2$, provides a splitting of $E$ into a sum of eigenbundles $E\approx L_1\oplus L_2$ for some $L_1,L_2\in{\rm Pic}(N)$, and taking an $S^1$ quotient of $M$ is equivalent to projectivizing the bundle $E$. The $S^1$ quotient that it produces is a $\bbc\bbp^1$ orbibundle $S_L$ over $N$ which we can write in terms of a log pair $(S_L,\grD)$. This orbibundle $S_L$ can be realized set theoretically as the projectivization
\begin{equation}\label{projbun}
M/S^1_\xi=E/\bbc^*_\xi=\bbp(\BOne\oplus L)= S_L
\end{equation}
where $L=L_2\otimes L_1^{-1}\in{\rm Pic}(N)$. Since the orbifold structure is on $\bbc\bbp^1$ it can be written as a log pair $(S_L,\grD)$ where $\grD$ is a branch divisor which is possibly empty. This proves (2).

The $\bbt^2$ action on $M$ gives a residual $S^1$ action on the orbifold $(S_L,\grD)$ which acts only on the $\bbc\bbp^1$ fibers over $N$ and which is automatically Hamiltonian proving (3).
\end{proof}

Now let $\gro_N$ be a K\"ahler form on $N$ such that its cohomology class $[\gro_N]$ is primitive in $H^{1,1}(N,\bbz)=H^2(N,\bbz)\cap H^{1,1}(N,\bbr)$. Define 
\begin{equation}\label{picsub}
{\rm Pic}(N,[\gro_N])=\{L\in {\rm Pic}(N)~|~c_1(L)=k[\gro_N] ~\text{for some $k\in\bbz$}\}.
\end{equation}
One easily checks that ${\rm Pic}(N,[\gro_N])$ is a subgroup of ${\rm Pic}(N)$. We have

\begin{thm}\label{mains3bunthm}
Let $M$ be an $S^3$-bundle over a smooth compact algebraic variety $N$, and let $\cald$ be a co-oriented contact structure on $M$ with an effective $\bbt^2$ action of Reeb type that acts trivially on $N$. Suppose also that the holomorphic line bundle $L$ of Lemma \ref{contactsub} lies in the subgroup ${\rm Pic}(N,[\gro_N])$ where the K\"ahler form $\gro_N$ is chosen to satisfy $\pi^*\gro_N=ad\eta |_{\calh\times\calh}$ for some $a\in\bbr^+$. Then $(M,\cald)$ has an almost regular Sasaki structure whose K\"ahler quotient is an orbifold represented as the log pair $(S_L,\grD)$ where $\grD$ is a branch divisor of the form
$$\grD=(1-\frac{1}{m})(D_0+D_\infty)$$
with ramification index $m$ where the divisors $D_0(D_\infty)$ denote the $0$ and infinity sections of $L$, respectively. Moreover, there is a choice of underlying CR structure $(\cald,J)$ which is cone decomposable.  In particular, there are positive integers $l_1,l_2,w_1,w_2$ with $\gcd(l_2,l_1w_i)=\gcd(w_1,w_2)=1$ such that $M$ is diffeomorphic to the smooth manifold arising from the join $M'\star_{l_1,l_2}S^3_\bfw$ where $\bfw=(w_1,w_2)$ and $M'$ is the Sasaki manifold corresponding to the principal $S^1$ bundle over $N$ with K\"ahler form $\gro_N$.
\end{thm}

\begin{proof}
Now we know by Lemma  \ref{contactsub} that $M$ is a non-trivial principal circle orbibundle over the ruled orbifold $(S_L,\grD)$. By hypothesis the K\"ahler class $[\gro_N]$ is fixed and $c_1(L)=n[\gro_N]$ for some $n\in\bbz$. As in Section 2.3 of \cite{BoTo14a} if $n>0$ we have a K\"ahler structure given by $n\gro_N$, whereas, if $n<0$ the K\"ahler structure is $-n\gro_N$. If $n=0$ we can choose the complex structure in the proof of Lemma \ref{contactsub} such that the quotient $S_L$ is the product $N\times \bbc\bbp^1$ in which case the Sasakian structure is decomposable. So we can restrict ourselves to the case that $n>0$. If we let $D_0$ denote the zero section of the line bundle $L$, then by the Leray-Hirsch Theorem the cohomology class corresponding to the $S^1$-bundle $M$ must take the form  
$\alpha+k_2h$, where $\alpha$ is a K\"ahler class pulled back from $N$ and $h$ is the Poincar\'e dual to $D$. Since the K\"ahler class $\gro_N$ on $N$ is chosen such that $\pi^*\gro_N=a d\eta |_{\calh\times\calh}$, we realize that $\alpha$ must be a multiple of the pullback of $[\omega_N]$ and so the cohomology class corresponding to the $S^1$-bundle $M$ takes the form  $k_1\pi^*_S[\gro_N]+k_2h$ with $k_1,k_2\in \bbz^+$ and we have a commutative diagram
\begin{equation}\label{s3s1comdia}
\begin{matrix}  M &&& \\
                          &\searrow{\pi_L} && \\
                          \decdnar{}\pi && S_L. &\\
                          &\swarrow{\pi_S} && \\
                         N &&& 
\end{matrix}
\end{equation}
Now since the $\bbt^2$ action on $M$ is the standard action on the fibers $F_x=S^3$, we see that writing $S^3$ as $|z_1|^2+|z_2|^2=1$ the action on the dense subset defined by $z_1z_2\neq 0$ is free for all $x\in N$. Consider the endpoints $z_2=0$ and $z_1=0$ which correspond to the divisors $D_0$ and $D_\infty$, respectively. The isotropy subgroups are denoted by $G_0$ and $G_\infty$, respectively. Note that $G_0$ and $G_\infty$ both contain an $S^1$ that is complementary to $S^1_\xi\subset \bbt^2$, and they may also contain  a finite cyclic subgroup of $S^1_\xi$. So generally we have $G_0\approx S^1\times \bbz_{m_0}$ and $G_\infty\approx S^1\times \bbz_{m_\infty}$ where $m_0,m_\infty\in\bbz^+$. However, if we choose the Reeb vector field $\xi_\bfv$ defined by $\bfv=(1,1)$ which is regular on each fiber, the quotient is invariant under the interchange of $D_0$ and $D_\infty$. It follows that the ramification indices are equal, that is $m_0=m_\infty=:m$. But this is the definition of an almost regular Reeb vector field \cite{BoTo14a}. In this case the fibers of $S_L$ are developable orbifolds of the form $\bbc\bbp^1/\bbz_m$. Now we claim that since $M$ is an $S^3$ bundle over $N$ the constraint $\gcd(n,m)=1$ holds.
To see this suppose $n$ and $m$ have a greatest common factor $k>1$. Then there exists a holomorphic line bundle $\call$ on $N$ such that $\call^k=L$. Setting $n=kn'$ and $m=km'$, we have the projective orbifold $(S_\call,\grD_{m'})$ where $S_\call=\bbp(\BOne\oplus \call)$ and $c_1(\call)=n'[\gro_N]$. The log pair $(S_\call,\grD_{m'})$ is a $k$-fold cover of the log pair $(S_L,\grD_{m})$. The corresponding primitive $S^1$ bundle over $(S_\call,\grD_{m'})$ is thus a $k$-fold cover of the primitive $S^1$ bundle over $(S_L,\grD_{m})$. This holds fiber wise and since $\gcd(n',m')=1$ we have the universal cover $S^3$ in this case. It follows that when $M$ is an $S^3$ bundle over $N$ the integers $n$ and $m$ must be relatively prime.
 
Now consider the join $M'\star_{l_1,l_2}S^3_\bfw$ as constructed in Section 3.2 of \cite{BoTo13,BoTo14a} where $M'$ is the unique positive primitive $S^1$ bundle over the Hodge manifold $(N,\gro_N)$. So by principal bundle theory there is a choice of the relatively prime pair $(l_1,l_2)$ and weight vector $\bfw$ such that $M$ and $M'\star_{l_1,l_2}S^3_\bfw$ are isomorphic as principal $S^1$-bundles over $S_L$ as long as the Reeb vector field on the join $M'\star_{l_1,l_2}S^3_\bfw$ giving rise to the quotient structure on $S_L$ is almost regular. So if we can identify the $S^1$ orbibundle from the construction in Diagram \eqref{s3s1comdia} with an appropriate $S^1$ orbibundle in the join, the corresponding orbifold Boothby-Wang constructions will identify the Sasakian structures up to a gauge transformation. The join depends on parameters $l_1,l_2,w_1,w_2$, whereas $M$ depends on parameters $k_1,k_2,m,n$. So we need to describe the relation between the two sets of parameters. This was essentially done in \cite{BoTo14a} at the end of Section 3. Specializing to the almost regular case we have $k_1=ml_1w_2$ and $k_2=l_2$ with $n=l_1(w_1-w_2)$ and $m=\gcd(k_1,k_2)$. As in \cite{BoTo14a} this determines $l_1,l_2,w_1,w_2$ uniquely. Moreover, since we know that $\gcd(n,m)=1$, Proposition \ref{harderreverse} below guarantees us that 
the identification of the parameters truly corresponds to identifying $M$ (with the chosen almost regular Reeb vector field) with the join $M'\star_{l_1,l_2}S^3_\bfw$ (with its almost regular Reeb vector field). This completes the proof of the theorem.
\end{proof}

\begin{rem} We chose to work with the almost regular Reeb vector field in the proof above, but we could have used any quasi-regular Reeb vector field (from the effective $\bbt^2$-action).
\end{rem}

\begin{rem}\label{nneq0rem}
In analogy to the choice of complex structure for the Sasakian structure in the case of $n=0$ in the above proof, in general there is also a choice of Sasaki CR structure in the $n \neq 0$ case that gives rise exactly to the join in its Sasaki cone and not a potentially twisted (in the complex structure sense) version of the join (making it only cone reducible). This twisted version arises from projective unitary reducible representations of $\pi_1(N)$, since ${\rm Pic}^0(N)$ acts on the set of line bundles of degree $n$.  In the case where $M$ is simply connected this is a non-issue since the long exact homotopy sequence implies that $N$ is also simply connected. 
\end{rem}

\begin{proof}[Proof of Theorem \ref{s3riemcor}]
If the Picard number $\grr(N)=1$, the subgroup ${\rm Pic}(N,[\gro_N])$ is all of ${\rm Pic}(N)$, so Theorem \ref{s3riemcor} follows directly from Lemma \ref{Alblem} and Theorem \ref{mains3bunthm}.
\end{proof}

\begin{proof}[Proof of Theorem \ref{Riesurfcor}]
By Theorem \ref{mains3bunthm} there is an almost regular Sasakian structure such that its $S^1$ quotient is the log pair $(S_L,\grD)$ with the branch divisor $\grD$ given in Theorem \ref{mains3bunthm} and with a holomorphic Hamiltonian action of $S^1$ for some line bundle $L\in {\rm Pic}(\grS_g)$. As in the beginning of the proof of Theorem \ref{mains3bunthm} there are essentially two cases $n=0$ and $n>0$. If the degree $n$ of the line bundle is zero, for each fixed complex structure $\grt$ on $\grS_g$ there is the Jacobian ${\rm Pic}^0(\grS_g)=T^{2g}$'s worth of complex structures. For $g>1$ these are parameterized by the singular part $\calr(\grS_g)^{sing}$ of the character variety $\calr(\grS_g)$, that is the reducible projective unitary representations $\grr$ of $\pi_1(\grS_g)$ in which case $S_L=\grS_g\times_\grr\bbc\bbp^1$ is a local product structure. Thus, the corresponding Sasakian structure on $M$ is reducible as described in Example \ref{rulsurfex}. 

If $n\neq 0$ by Theorem \ref{mains3bunthm} for each complex structure $\grt$ on $\grS_g$ there is a choice of CR structure $(\cald,J)$ on $M$ such that $(\cald,J)$ is cone decomposable. This corresponds to a choice of line bundle $L\in {\rm Pic}(\grS_g)$ of degree $n$, and a choice of splitting of the exact sequence 
$$0\ra{2.0}{\rm Pic}^0(\grS_g)\ra{2.0} {\rm Pic}(\grS_g)\fract{c_1}{\ra{2.0}}\bbz\ra{2.0} 0.$$
Since ${\rm Pic}^0(\grS_g)$ acts transitively on the set of line bundles of degree $n$, any other line bundle $L'\in {\rm Pic}(\grS_g)$ of degree $n$ is obtained from $L$ by an element $\grr\in {\rm Pic}^0(\grS_g)$. Now we know from the proof of Theorem \ref{mains3bunthm} that with $n=b_2-b_1$ our choice of line bundle gives $\bfb=l_1(w_1-w_2)$ which corresponds to the decomposable Sasakian structure over the quotient $\grS_g\times \bbc\bbp^1(\bfw)$. Moreover, since changing $n$ to $-n$ does not change the quotient orbifold nor the Sasaki CR structure, we restrict ourselves to the case $n>0$. Thus, when $g>1$ for any $\grr\in {\rm Pic}^0(\grS_g)\approx T^{2g}$ as in Example \ref{rulsurfex} we can construct the twisted quotients $\grS_g\times_\grr \bbc\bbp^1(\bfw)$. Since these are local products, the orbifold Boothby-Wang construction gives reducible Sasakian structures on the corresponding $S^3$ bundle over $\grS_g$ that are decomposable only if $\grr=\{id\}$.

In the $g=1$ case we use the results of Suwa \cite{Suw69,BoTo11}. When $n=0$ the discussion on pages 294-295 of \cite{Suw69} implies that there are precisely a $\bbc\bbp^1$'s worth of complex structures for each complex structure on $\grS_g$. Moreover, as the $g>1$ case these give rise to local product structures, so the Sasaki CR structure is cone reducible with the one corresponding to the trivial line bundle being decomposable. However, when $n\neq 0$ it follows from Lemma 1 of \cite{Suw69} that any reducible Sasakian structure is equivalent to a decomposable Sasakian structure, since the Jacobian of a $\grS_1$ is equivalent to $\grS_1$ itself .
\end{proof}

\begin{rem}\label{lensrem}
The general case of the $S^3_\bfw$-join treated in \cite{BoTo14a} gives 3-dimensional lens space bundles over $N$. So it is natural to ask whether Theorem \ref{s3riemcor} would hold in this more general setting. Unfortunately, there is no analogue of Eliashberg's uniqueness of tight contact structure for general lens spaces. In fact, there are many tight contact structures on lens spaces even those that lift to a tight contact structure on the universal cover $S^3$ \cite{Hon00a,Gir00} and they can have a toric structure \cite{Ler02a}. There is, however, a unique tight contact structure on the lens space $L(2,q,1)$ \cite{Etn00}, which is diffeomorphic to $\bbr\bbp^3$, and it must be the standard structure, so Theorem \ref{s3riemcor} will hold for this case which corresponds to the join $M\star_{1,2}S^3_{q,1}$ with $q$ odd.
\end{rem}

\begin{rem}\label{Kahrem}
The fact that $\gro_S$ is a K\"ahler form can put restrictions on the integers $n,l_1,w_1,w_2$ as we shall see with the example below.
\end{rem}

The above proof shows that in the genus one Riemann surface case the hypothesis that the toral action be of Reeb type is needed. It is not needed in the genus zero case since it is known that all toric contact structures on an $S^3$ bundle over $S^2$ are of Reeb type \cite{BoPa10}. However, in the case of a $\bbt^2$ action on $S^3$ bundles over Riemann surfaces of genus greater than one, it is not known. For example, there are the well known overtwisted contact structures $\cald_{ot}$ on $S^3$ due to Eliashberg  \cite{Eli89}. There is a unique overtwisted contact structure $\cald_{ot}$ on $S^3$ with vanishing Hopf invariant. Using contact cuts Lerman \cite{Ler01} constructed an infinite sequence of toric overtwisted contact structures on $S^3$ with vanishing Hopf invariant that are $\bbt^2$ equivariantly inequivalent which by Eliashberg are contact equivalent. It follows that $(S^3,\cald_{ot})$ has an infinite number of toric contact structures and that $\gC\go\gn(S^3,\cald_{ot})$ has an infinite number of 
conjugacy classes of maximal tori (see also Example 7.14 in \cite{Boy10a}). However, we do not know whether these contact structures on the fiber $S^3$ can extend to a contact structure on the whole $S^3$ bundle over $\grS_g$ that is $\bbt^2$ invariant. It is known from Lerman's classification \cite{Ler02a} that this type of extension cannot happen in the toric case.

In \cite{BoTo13} it is shown that there are a countable infinity of contact structures $\cald_k$ with $k\in\bbz^+$ of Sasaki type on both $\grS_g\times S^3$ and $\grS_g\tilde{\times}S^3$ for $g>0$ essentially labeled by the first Chern class of the contact bundle. For each such $k$ there are $k$ two dimensional Sasaki cones with a unique ray of constant scalar curvature Sasaki metrics. Moreover, in the case of the trivial bundle 
$\grS_g\times S^3$ it is shown using the work of Bu{\c{s}}e \cite{Bus10} on equivariant Gromov-Witten invariants that the $k$ Sasaki cones belong to inequivalent $\bbt^2$ equivariant contact structures that are contact equivalent, and so the Sasaki cones form a bouquet of Sasaki cones \cite{Boy10a}. Moreover,  for $g\geq 2$ one can twist the transverse complex structures with reducible representations of the fundamental group $\pi_1(\grS_g)$ giving the $k+1$-bouquet $\gB_{k+1}(\cald_k)$ described by Theorem 4.5 of \cite{BoTo13}. Theorem \ref{s3riemcor} implies that $\gB_{k+1}(\cald_k)$ is a complete bouquet for $\cald_k$. We note that the topology of this bouquet is non-Hausdorff.

\subsection{Examples of Cone Indecomposability}
Here we give some examples of cone indecomposable Sasakian structures using a construction of Yamazaki \cite{Yam99}. We have

\begin{prop}\label{yamexprop}
Let $N$ be a smooth compact projective algebraic variety with Picard number $\grr(N)>1$. There is an $S^3$-bundle $M$ over $N$ with a co-oriented contact structure $\cald$ and an effective $\bbt^2$ action of Reeb type that acts trivially on $N$ and with an induced Sasakian structure on $M$ that is cone indecomposable. 
\end{prop}

\begin{proof}
First there is a construction due to Yamazaki \cite{Yam99} for constructing such Sasaki manifolds. We begin with the projective algebraic variety\footnote{Yamazaki actually works in the symplectic and K-contact categories, but it is straightforward to see that under the correct circumstances his construction easily adabts to the Sasaki category.} $N$. Since $\grr(N)>1$, the K\"ahler cone has dimension at least $2$. Let $[\gro_j]$ be two primitive integral K\"ahler classes with $j=1,2$ which are not multiples of each other. Let $M_j$ be the principal $S^1$ bundles over $N$ whose Euler class is $b_j[\gro_j]$ for some $b_j\in\bbz^+$. Then it is straightforward to see that $M_j$ are regular Sasaki manifolds. Now let $L_j$ be holomorphic lines bundles associated to $M_j$. Then Yamazaki's `fiber join' $M_1*_fM_2$ is the unit sphere bundle $M=S(L_1\oplus L_2)$ which as Yamazaki shows has a co-oriented contact structure $\cald$ with an effective $\bbt^2$ action of Reeb type that acts trivially on $N$. Thus, we have an $S^3$ bundle over $N$ with a two dimensional Sasaki cone $\gt^+_2$. Now choose the unique almost regular Reeb vector field in $\gt^+_2$ and consider the quotient $S_L$. It must take the form $S_L=\bbp(L_1\oplus L_2)=\bbp(\BOne\oplus L)$ with $L=L_2L_1^{-1}$. We have 
$$c_1(L)=c_1(L_2)-c_1(L_1)=b_2[\gro_2]-b_1[\gro_1]$$
which cannot be $\pm$(K\"ahler class) for all $b_1,b_2\in\bbz^+$. Such a choice gives $M$.

We now show that $M$ is cone indecomposable.
Assume to the contrary that $M$ is equivalent to a join $M'\star_{l_1,l_2}S^3_\bfw$ for some choice of quasi-regular Reeb vector field. First we note that an equivalence of Sasakian structures implies, in the (quasi)-regular case, an equivalence of the projective algebraic quotients. Then since $M$ is the join $M'\star_{l_1,l_2}S^3_\bfw$, Equation (33) of \cite{BoTo14a} (with the order of $z_1,z_2$ reversed) implies that the $\bbc\bbp^1$ bundle $S_L$ is an associated bundle to the principal $S^1$ bundle $M'$ over $N$. Thus, $c_1(L^*)=n[\gro_N]$ for some $n\in\bbz\setminus \{0\}$. This gives a contradiction.
\end{proof}

\begin{example}\label{yamex}
As an explicit example of Proposition \ref{yamexprop} we take $N$ to be a product of Riemann surfaces with genera $g_1,g_2$, respectively, i.e. $N=\grS_{g_1}\times \grS_{g_2}$. Let $\gro_i$ be the standard K\"ahler form on $\grS_{g_i}$, and consider the K\"ahler forms on $\grS_{g_1}\times \grS_{g_2}$ given by $c_1\gro_1+\gro_2$ and $\gro_1+c_2\gro_2$ with $c_1,c_2\in\bbz^+$. Then the principal $S^1$ bundles $M_1$ and $M_2$ over $\grS_{g_1}\times \grS_{g_2}$ with K\"ahler forms $c_1\gro_1+\gro_2$ and $\gro_1+c_2\gro_2$, respectively, have distinct natural Sasakian structures. Let $L_1,L_2$ denote the associated complex line bundles to the principal bundles $M_1,M_2$, respectively, and set $L=L_2\otimes L_1^{-1}$. Then Proposition \ref{yamexprop} says that the fiber join $M_1*_fM_2$ has a Sasakian structure with $c_1(L^*)=-(c_1-1)[\gro_1]+(c_2-1)[\gro_2]$ which is not a K\"ahler class for any pair $(c_1,c_2)\in(\bbz^+)^2\setminus\{(1,1)\}$. So this gives infinitely 
many cone irreducible Sasakian structures. The pair $(c_1,c_2)=(1,1)$ is cone reducible. In particular, taking $g_1=g_2=0$ we get infinitely many cone irreducible toric contact structures on $S^3$ bundles over $S^2\times S^2$.
\end{example}

The next example shows Theorem \ref{Riesurfcor} does not hold in the genus $g=0$ case. In fact there are infinitely many inequivalent contact structures\footnote{The fact that the contact structures $L(2,2,2,2k)$ are inequivalent for different $k$ has recently been proven by Uebele \cite{Ueb15} using the plus part of non-equivariant symplectic homology on a convenient filling which Uebele shows is a contact invariant in this case.  Interestingly these contact structures cannot be distinguished by their mean Euler characteristic nor the plus part of the $S^1$-equivariant symplectic homology which are known contact invariants \cite{KwvKo13,BMvK15}.} of Reeb type with a two dimensional Sasaki cone on $S^2\times S^3$  which we show are cone irreducible. 

\begin{example}\label{BPlink}
Consider the Brieskorn-Pham link $L(2,2,2,2k)$ of degree $2k$ defined by 
$$L(2,2,2,2k)=\{z_0^2+z_1^2+z_2^2+z_3^{2k}=0\}\cap S^7.$$
For all positive integers $k$ the link $L(2,2,2,2k)$ is diffeomorphic to $S^2\times S^3$.  Concerning reducibility we have

\begin{prop}\label{222coneirr}
The link $L(2,2,2,2k)$ is cone irreducible for all $k\geq 2$.
\end{prop}

\begin{proof}
First we note that the link $L(2,2,2,2k)$ has a two dimensional Sasaki cone $\gt^+_2$. So suppose that $M=L(2,2,2,2k)$ is cone reducible. Then there is a quasi-regular Reeb vector field $\xi_1$ in $\gt^+_2$ whose quotient $\calz$ is a local product of one dimensional complex orbifolds. However, since $L(2,2,2,2k)$ is simply connected we must have $\pi_1^{orb}(\calz)=\{id\}$ by the long exact homotopy sequence. Thus, by Lemma \ref{orbdeRh} the orbifold $\calz$ is a product $\calz=\calo_1\times \calo_2$ of one dimensional complex orbifolds $\calo_i$ for $i=1,2$ with $\pi_1^{orb}(\calo_i)=\{id\}$; hence, $M$ is cone decomposable. Since $\pi_1^{orb}(\calo_i)=\{id\}$, the orbifold $\calo_i$ is not developable (good in Thurston's terminology \cite{Thu79}). It follows that $\calz$ is a product of weighted projective $\bbc\bbp^1$'s, i.e. $\calz=\bbc\bbp^1(\bfu)\times \bbc\bbp^1(\bfv)$ with $\bfu\neq (1,1), \bfv\neq(1,1)$. But then both orbifolds $\calo_i$ have an $S^1$ Hamiltonian symmetry, so the Sasaki cone of $M$ would have dimension $3$, not $2$. This gives a contradiction.
\end{proof}

Let us examine a bit closer what the quotient orbifolds might look like.
It is important here that reducibility means that one has a product of one dimensional complex {\it orbifolds}. The link $L(2,2,2,2k)$ has degree $2k$ with weight vector $\bfw=(k,k,k,1)$. So as projective algebraic varieties we have an embedding of the zero locus of the weighted homogeneous polynomial $z_0^2 +z_1^2+z_2^2+z_3^{2k}$ in the non well-formed weighted projective space $\bbc\bbp^3[k,k,k,1]$. So the divisor $z_3=0$ is a branch divisor with ramification index $k$. Furthermore, the map $(z_0,z_1,z_2,z_3)\mapsto (z_0,z_1,z_2,z_3^k)$ gives an isomorphism of projective algebraic varieties $\bbc\bbp^3[k,k,k,1]\approx \bbc\bbp^3$ and the zero locus of  $z_0^2 +z_1^2+z_2^2+z_3^{2k}$ with the zero locus of the quadric $z_0^2 +z_1^2+z_2^2+z_3^{2}$. The quadric in $\bbc\bbp^3$ is isomorphic to $\bbc\bbp^1\times \bbc\bbp^1$ by the well known Segre embedding 
$$([x_0,x_1],[y_0,y_1])\mapsto [x_0y_0,x_0y_1,x_1y_0,x_1y_1]=[u_0,u_1,u_2,u_3].$$
The image of $\bbc\bbp^1\times \bbc\bbp^1$ in $\bbc\bbp^3$ is given by the quadric $u_0u_3=u_1u_2$. So we make the change of variables
$$u_0=z_0+iz_1,~u_3=z_0-iz_1,~u_1=iz_2+z_3,~u_2=iz_2-z_3$$
so that $u_0u_3-u_1u_2=z_0^2+z_1^2+z_2^2+z_3^2$ and the divisor $z_3=0$ becomes $u_1=u_2$. The latter is equivalent to $[y_0,y_1]=[x_0,x_1]$. Thus, the quotient of the link $L(2,2,2,2k)$ is isomorphic to the log pair 
\begin{equation}\label{logpair}
(\bbc\bbp^1\times \bbc\bbp^1,(1-\frac{1}{k})\grD)
\end{equation} 
where the divisor $\grD$ is the diagonal embedding $\bbc\bbp^1\ra{1.6} \bbc\bbp^1\times \bbc\bbp^1$. This shows that although the quotient of $L(2,2,2,2k)$ by the $S^1$ action generated by the standard Reeb vector field $\xi_\bfw$ is a product of algebraic varieties, it is not a product of orbifolds if $k\geq 2$, so it cannot arise from a join. 

The connected component of the Sasaki automorphism group is $SO(3)\times U(1)$ where the $U(1)$ is generated by the Reeb vector field. So we can choose a $\bbt^2$ subgroup as the $U(1)\times SO(2)$ where the $SO(2)$ can be taken to be real rotations in the $z_0,z_1$-plane, that is by the matrix 
$$\begin{pmatrix} \cos\theta & \sin\theta \\
                            -\sin\theta & \cos\theta
                            \end{pmatrix}. $$
The $SO(2)$ action on $\bbc\bbp^1\times \bbc\bbp^1$ is then given by
$$([x_0,x_1],[y_0,y_1])\mapsto ([e^{-i\frac{\theta}{2}}x_0,e^{i\frac{\theta}{2}}x_1],[e^{-i\frac{\theta}{2}}y_0,e^{i\frac{\theta}{2}}y_1]).$$
Thus, $\bbt^2$ acts on both the base and the fiber.

We remark that in the case $k=1$ the orbifold structure is trivial and we obtain the well known homogeneous Sasaki-Einstein structure on $S^2\times S^3$ which is indeed decomposable and toric. 
\end{example}

Finally we mention that for $k\geq 2$ it is well know that the link  $L(2,2,2,2k)$ with its standard Reeb vector field $\xi_\bfw$ where $\bfw=(k,k,k,1)$ does not admit a Sasaki-Einstein metric by the Lichnerowicz obstruction \cite{GMSY06,BG05}.  But since in this case the Lie algebra $\gh_0$ (see \cite{BGS06} for the definition) is the simple Lie algebra $\gs\go(3,\bbc)$, the Sasaki-Futaki invariant $\gF$ vanishes identically. So this link does not admit any extremal Sasaki metric as well. It is still an open question\footnote{Note added: this question has been recently answered in the negative in \cite{BovCo16}.} whether there are extremal Sasaki metrics in the Sasaki cone $\gt^+_2$, but there are no Sasaki metrics of constant scalar curvature in $\gt^+_2$ by \cite{MaSpYau06} in the quasi-regular case and \cite{He14} in the irregular case.

\subsection{Proof of Theorem~\ref{theoSPLITTINGtoricSimplex}}

There is one-to-one correspondence between toric contact manifolds $(M^{2n+1},\cald,\bT)$ and good rational cones $C$ in $\kt^*=(\mbox{Lie} \bT)^*\simeq \R^{n+1}$ as established in~\cite{BM93,BG00b, Ler02a}. Recall from Subsection \ref{conesect} that the moment cone $C$ is the image of the (order $2$ homogenous) moment map $\Upsilon: \cald^o_+ \rightarrow \kt^*$ of the symplectisation of $(M^{2n+1},\cald, \bT)$ where the action of $\bT$ is the natural lift by pull-back and commutes with the $\R_+$ action of the cone. {\it Note that the cone $C$ does not contain the origin but its closure does.} 

The existence of a compatible toric Sasaki structure, and more particularly the existence of a Reeb vector field commuting with the action of $\bT$, implies that the moment cone is {\it strictly convex}, equivalently its dual cone $$C^* = \{ y \in \R^{n+1}\,|\, \langle x,y \rangle >0 \;\; \forall x \in C\}$$ has a non empty interior, namely the Sasaki cone $\gt^+$. This is very explicit, the cone $C^*$ parametrizes the set of toric (i.e $\bT$--invariant) Reeb vector fields. Indeed, recall that each $b\in\gt^+\subset C^* \subset \kt =\mbox{Lie} \bT $, induces a vector field $X_b$ on $M$ via the action and given any $\bT$--invariant contact form $\eta$ on $(M,\cald)$, the function $\eta(X_b) =\langle \mu_\eta, b\rangle >0$ is positive and the $1$--form $$\eta_b := \frac{1}{\eta(X_b)}\eta$$ is a $\bT$--invariant contact form whose Reeb vector field is $X_b$.    

In view of Remark~\ref{remLABELgoodCON} and Proposition~\ref{condCarcGC}, it is convenient to label polytopes with their defining affine functions instead of the normals (or integers). We adopt this convention in this section.

The standard $n$-simplex is $\Delta_n:=\{ x=(x_1,\dots, x_n)\in \R^n\,|\, x_i\geq 0, \sum_{i=1}^n x_i\leq 1\}$ and a $n$--simplex is any polytope affinely equivalent to $\Delta_n$. Two polytopes have the same combinatorial type if there is a bijection between their faces that preserves the relation of inclusion. Observe that this is the case of any two compact polytopes characteristic (or transversal) to the same polyhedral cone. In particular it makes sense to speak about the combinatorial type of a cone (without vertex) and compare it to that of a polytope.   

\begin{lemma}\label{lemmaSPLITTINGtoricSimplex} Let $(C^{n_1+n_2+1},\Lambda)$ be a strictly convex good polyhedral cone such that its characteristic polytopes has the combinatorial type of $\Delta_{n_1}\times \Delta_{n_2}$. Then there exists $b\in \Lambda\cap C^*$ such that $C\cap \{x \,|\, \langle x, b \rangle =1\}$ is a product of two simplices $P_1\times P_2$.    
\end{lemma}

From this Lemma we easily get the following statement from which Theorem~\ref{theoSPLITTINGtoricSimplex} is extracted.   
\begin{cor}\label{coroSPLITTINGtoricSimplex} Let $(M,\cald,T)$ be a toric contact manifold of Reeb type whose moment cone has the combinatorial type of a product of $n_1$ and $n_2$ dimensional simplices. Then $(M,\cald)$ is cone reducible and there exists a Reeb vector field $X$ for which $(M,\cald, X)$ is obtained as the join construction of two weighted projective spaces of complex dimension $n_1$ and $n_2$, respectively.  
\end{cor}

We will need the following lemma the proof of which is nearly trivial. 
\begin{lemma}\label{lemmaSPLITTINGlabelling} A labelled polytope $(P,u)$ is a product if and only if one can split the set of normals in two disjoint subsets $u= \{u_i\}_{i=1}^{d_1}\cup \{u_i\}_{i=d_1+1}^{d=d_1+d_2}$ such that $$\sum_{i=1}^d x_iu_i = 0\implies \sum_{i=1}^{d_1} x_iu_i =0  \;\;\mbox{ and }\;\;\sum_{i=d_1+1}^{d_1+d_2} x_iu_i=0.$$    
\end{lemma}

\begin{proof}[Proof of Lemma~\ref{lemmaSPLITTINGtoricSimplex}]
The hypothesis implies that we can split the set of normal inward vectors of $C$ in two groups \begin{equation}\label{normalsSimplProd}
l_0^1,\dots, l_{n_1}^1 \;\;\mbox{ and }\;\; l_0^2,\dots, l_{n_2}^2
\end{equation}
such that there is no edge of $C$ on which every vector of one group vanishes (i.e it would correspond to a vertex lying in every facet of a simplex). 

To prove the lemma it suffices to find a Reeb vector $b$ lying in $\Lambda$ and that is a linear combination of the $\{l_i^1\}_{i=0}^{n_1}$ and a linear combination of the $\{l_i^2\}_{i=0}^{n_2}$. Indeed, in the quotient space $\kt /\R b$ the set of vectors $\{[l_0^1],\dots, [l_{n_1}^1]\}$ is linearly dependant as well as the set $\{[l_0^2],\dots, [l_{n_2}^2]\}$. Consequently, by a dimensional argument (this is where the simplices assumption comes in) the characteristic labelled polytope of $(C,\Lambda)$ at $b$ satisfies Lemma~\ref{lemmaSPLITTINGlabelling} and is thus a product.       

We introduce some notation. Let $F_i^\epsilon$ denote the facet of $C$ corresponding to the zero locus (in $C$) of $l_i^\epsilon$ for $\epsilon \in \{1,2\}$ and $i\in I_\epsilon=\{0,\dots, n_\epsilon\}$. Edges of $C$ are parametrized by $(i,j) \in I_1\times I_2$ so that the corresponding edge is $$E_{(i,j)} =\left(\bigcap_{k \in I_1\backslash \{ i\}} F_k^1\right)\cap \left(\bigcap_{k \in I_2\backslash \{ j\}} F_k^2\right).$$
First, we will prove that the open positive cones generated by each set of normals~\eqref{normalsSimplProd} have to meet in the cone $C^*$. Indeed, these cones $$C_1^*=\Span_{\R_{>0}}\{l_0^1,\dots, l_{n_1}^1\}\;\; \mbox{ and }\;\; C_2^*=\Span_{\R_{>0}}\{l_0^2,\dots, l_{n_2}^2\}$$ are respectively of dimension $n_1+1$ and $n_2+1$ in a space of dimension $n_1+n_2+1$. Hence the linear subpaces they generate meet in a line (at least) which contain a non trivial vector, say $b$. That is there exists for each $\epsilon \in \{1,2\}$ a vector $(a_0^\epsilon, \dots a_{n_\epsilon}^\epsilon)\in \R^{n_\epsilon+1} \backslash \{0\}$ such that $$b= \sum_{i=0}^{n_\epsilon} a_i^\epsilon l_i^\epsilon.$$  Now pick a point $x$ lying in the edge $E_{(i,j)}$ and evaluate $b$ on it. We have $$\langle b,x\rangle = a_i^1 \langle l_i^1, x \rangle =  a_j^2 \langle l_j^2, x \rangle$$ but $\langle l_i^1, x \rangle$ and $\langle l_j^2, x \rangle$ are both positive so that $a_i^1$ and $a_j^2$ have the same sign or both vanish. Since 
we can do the same argument for all $(i,j) \in I_1\times I_2$ then $b$ or $-b$ lies in $C_1^*\cap C_2^* \subset C^*$. 

To conclude the proof we need to prove that we can pick $b\in \Lambda \cap C_1^*\cap C_2^*$, so that the characteristic labelled polytope associated to $b$ is rational. The vectors above $(a_0^\epsilon, \dots a_{n_\epsilon}^\epsilon)\in \R^{n_\epsilon+1}$ satisfy $a^1_i/a^2_j \in \Q$ for each pair $(i,j) \in I_1\times I_2$. This relation is a consequence of the hypothesis that $(C,\Lambda)$ is good. Indeed, $E_{(i,j)}\cap \Lambda$ is then non-empty and for $x\in E_{(i,j)}\cap \Lambda$ we have $a_i^1/a_j^2 =\langle l_j^2, x \rangle/\langle l_i^1, x \rangle\in \Q.$ Hence, $a_i^1/a_j^1 \in \Q$ for $i,j\in I_1$ and then, up to an overall factor $(a_0^\epsilon, \dots a_{n_\epsilon}^\epsilon)\in \Z^{n_\epsilon+1}$.   \end{proof}

\subsection{Reversing the quotient of a Join}\label{join reverse}
It was shown in \cite{BoTo14a} how one begins with a certain product of projective algebraic orbifolds and constructs cone reducible Sasakian structures such that any quasi-regular Sasaki structure in the $\bfw$-cone is an orbibundle over a log pair $(S_n,\grD)$ consisting of a ruled manifold $S_n=\bbp(\BOne\oplus L_n)$ together with a certain branch divisor $\grD$. It is the purpose of this section to invert this procedure. 
In the following we assume that a K\"ahler form $\omega_N$ with primitive K\"ahler class $[\omega_N]$ has been fixed on a compact K\"ahlerian manifold $N$. Then by $S_n$ we mean the total space of $\bbp(\BOne\oplus L_n)$, where $L_n \rightarrow N$ is a holomorphic vector bundle such that $c_1(L_n) = n[\omega_N]$. Likewise, an
$S^3_\bfw$-join $M_{l_1,l_2,\bfw}$ is assumed to use the chosen $\omega_N$.

For the special case where $l_2=1$, Theorem 3.8 in \cite{BoTo14a} tells us that the quotient of the $S^3_\bfw$-join $M_{l_1,1,\bfw}$ by the flow of the Reeb vector field $\xi_\bfv$, determined by co-prime $v_1,v_2 \in \bbz^+$, in the $\bfw$-cone is the log pair $(S_n,\grD)$
where $\Delta$ denotes the branch divisor
$$\grD= (1-1/v_1)D_1+(1-1/v_2)D_2,$$
where $D_1,D_2$ are the zero, infinity sections of $S_n$, respectively and $n=l_1(w_1 v_2 - w_2 v_1)$. 
For convenience we will introduce the notation $(S_n,\grD_{v_1,v_2})$
It is natural to ask if all such log pairs may arise as such a quotient.
In the following we will not make the assumption $w_1\geq w_2$. This is a practical assumption made in \cite{BoTo14a}, but it is not being used in the arguments leading up to Theorem 3.8 in \cite{BoTo14a}.  The only reason for making this assumption in \cite{BoTo14a} was to avoid redundancy and it just parallels the fact that
$(S_n,\grD_{v_1,v_2}) \cong (S_{-n},\grD_{v_2,v_1})$.

\begin{prop}\label{easyreverse}
For any choice of $n\in \bbz$ and co-prime $v_1,v_2 \in \bbz^+$, there is a choice of co-prime $w_1,w_2\in \bbz^+$ such that
the quotient of  $M_{|n|,1,\bfw}$ by the flow of the Reeb vector field $\xi_\bfv$, determined by co-prime $(v_1,v_2)$, in the $\bfw$-cone is the log pair 
$(S_n,\grD_{v_1,v_2})$.
\end{prop}

\begin{proof}
The case where $n=0$ is trivial, so we assume that $n\in \bbz\setminus\{0\}$.
The proof is simply using the fact (following from B\'ezout's Identity) that for co-prime $v_1,v_2\in\bbz^+$ we can always find co-prime $a,b \in \bbz^+$ such that
$$av_2+ bv_1 = \frac{n}{|n|}$$
and since $v_1,v_2 \in \bbz^+$, we realize that the integers $a$ and $b$ must have opposite signs. 
Indeed we may assume that $b<0$ (by adding $kv_1$ to $a$ and $-kv_2$ to b for a sufficiently large $k\in \bbz$).  If we now let  $w_1=a$, and $w_2=-b$, then with
$l_1=|n|$, we have $l_1(w_1v_2-w_2v_1) =n$ and thus the result follows.
\end{proof}

\begin{rem}
As is clear from the proof above, $\bfw=(w_1,w_2)$ are by no means unique. Indeed we have (at least) a countable infinite set of choices
$(w_1^i,w_2^i)$, $i=1,2,\dots$, where $\displaystyle\lim_{i \rightarrow +\infty} w_2^i = +\infty$. From Lemma 3.11 of \cite{BoTo14a} (with $k_2=1$, $m_1=v_1$, and $l_1=|n|$) we
have the corresponding primitive
K\"ahler classes induced on $(S_n,\grD_{v_1,v_2})$:
$$\Omega = |n| v_1 w_2^i p_{\bfv}^*[\omega_N] + PD(D_1)$$
Here $p_{\bfv}$ denotes the projection from $(S_n,\grD_{v_1,v_2})$ to $N$ and $PD$ denotes the Poincar\'e dual.

Clearly no two distinct viable choices of $(w_1,w_2)$ in Proposition \ref{easyreverse} will result in the same K\"ahler class.
\end{rem}

More generally, let 
$$\Omega = k_1p_{\bfv}^*[\omega_N] + k_2PD(D_1)$$
denote a specific (so-called {\em admissible}) primitive K\"ahler class on $(S_n,\grD_{m_1,m_2})$, where
\begin{equation}\label{mbranch}
\grD_{m_1,m_2}= (1-1/m_1)D_1+(1-1/m_2)D_2
\end{equation}
(with $\gcd(m_1,m_2)=m$ not necessarily equal to one),
and $k_1,k_2 \in \bbz^+$ such that $k_1/k_2 >-n$. 
Then we have a natural Sasaki structure given by the orbifold Boothby-Wang construction. This Sasaki manifold  is a $S^1$-orbibundle over $(S_n,\grD_{m_1,m_2})$ which may or may not be a smooth manifold.
A more subtle question to consider is the following: When does such a  Sasaki structure correspond to a ray in the $\bfw$-cone of a $S^3_\bfw$-join?
In other words, when do we have a $S^3_\bfw$-join which can be obtained from the orbifold Boothby-Wang construction using 
$(S_n,\grD_{m_1,m_2})$ and $\Omega = k_1p_{\bfv}^*[\omega_\Sigma] + k_2PD(D_1)$? As we will see below, given a natural assumption, the answer is {\it always}.

Using Theorem 3.8  and Lemma 3.11 in \cite{BoTo14a} we have the following algorithm for determining the necessary values of  $(w_1,w_2,l_1,l_2)$:

\bigskip

\begin{enumerate}
\item Let $r$ be such that $\frac{n(1-r)}{2r} = \frac{k_1}{k_2}$. Note $0<|r|<1$ and $r$ has the same sign as $n$.
Now $\bfw = (w_1,w_2)$ is the unique positive, integer, and co-prime solution of
$$
r=\frac{w_1m_2-w_2m_1}{w_1m_2+w_2m_1}.
$$
\item Using these $w_1$ and $w_2$, the pair $(l_1,l_2)$ has to be the unique positive integers, and co-prime solution of
$$
\l_2 n = l_1(w_1m_2-w_2m_1)
$$
\end{enumerate}

\bigskip

Using $w_1,w_2,l_1,l_2$ from this  algorithm , we then have that
the Sasaki structure corresponds to a ray in the $S_\bfw$-cone of the (possibly non-smooth) $S^3_\bfw$-join $M_{l_1,l_2,\bfw}$  
if and only if 
\begin{equation}\label{joincond}
l_2 = \gcd(ml_2, |w_1m_2-w_2m_1|).
\end{equation}
In that case, $M_{l_1,l_2,\bfw}$ is a smooth manifold if and only if 
\begin{equation}\label{smoothjoincond}
\gcd (w_1,l_2)=\gcd(w_2,l_2) = 1.
\end{equation}

\begin{prop}\label{harderreverse}
For any choice of $n\in \bbz$, $m_1,m_2 \in \bbz^+$ such that  $n= 0$ or $\gcd(m_1,m_2,n)=1$, and primitive K\"ahler class on $(S_n,\grD_{m_1,m_2})$ of the form
$$\Omega = k_1p_{\bfv}^*[\omega_N] + k_2PD(D_1),$$ 
there is a unique choice of co-prime $w_1,w_2\in \bbz^+$ and co-prime $l_1,l_2\in \bbz^+$ such that,
when we form the $S^3_\bfw$-join $M_{l_1,l_2,\bfw}$,
the quotient of  $M_{l_1,l_2,\bfw}$ by the flow of the Reeb vector field $\xi_\bfv$, determined by $(\frac{m_1}{\gcd(m_1,m_2)},\frac{m_2}{\gcd(m_1,m_2)})$ in the $\bfw$-cone, is the log pair 
$(S_n,\grD_{m_1,m_2})$ with induced K\"ahler class $\Omega$.
This join is smooth if and only if \eqref{smoothjoincond} holds. In particular, the join is smooth if $m_1=m_2=1$.
\end{prop}

\begin{proof}
The case where $n=0$ is trivial, so 
let $n\in \bbz\setminus\{0\}$ and $m_1,m_2,k_1,k_2 \in \bbz^+$ be given so that $k_1/k_2 >-n$ and $\gcd(n,m)=1$, where $m=\gcd(m_1,m_2)$.
We define $v_i=m_i/m$ for $i=1,2$. 
Suppose we have determined $w_1,w_2,l_1,l_2$ following the algorithm above. We want to prove that \eqref{joincond} holds.

Let $p = \gcd(n,|w_1v_2-w_2v_1|)$. Since $\gcd(n,m)=1$ we know that $\gcd(p,m)=1$. Further
$w_1v_2-w_2v_1= p q$ for some $q \in \bbz$ satisfying that $\gcd(|q|,n)=\gcd(|q|,\frac{n}{p})=1$.

From step (2) in the algorithm we know that co-prime $l_1$ and $l_2$ satisfy that
$$
\l_2n = l_1m p q
$$
and hence
$$
l_2\frac{n}{p} = l_1m q
$$
Using the observations above, this tells us that $l_1=\frac{|n|}{p}$ and $l_2=m \,|q|$.

Now, since $\gcd(p,m) = 1$,
$$gcd(ml_2,|w_1m_2-w_2m_1|) = \gcd(m^2\,|q|, mp |q|) =m |q|,$$
and so \eqref{joincond} is satisfied.

The smooth case of $m_1=m_2=1$ is straightforward to verify.
\end{proof}

\begin{rem}
To see that the assumption  $\gcd(m_1,m_2,n)=1$ in the case where $n\neq 0$ is not trivial one may for instance work out the algorithm for
$n=m_1=m_2=2$, $k_1=4$, and $k_2=1$. This gives us $r=1/5$, $w_1=3$, $w_2=2$, and $l_1=l_2=1$. It is easy to see that this does not
satisfy \eqref{joincond}. 
\end{rem}

\begin{rem}
Recall that Proposition \ref{yamexprop} and Example \ref{yamex} show that there are ruled manifolds $S_L=\bbp(\BOne\oplus L)$ with Sasakian $S^1$ bundles that do not arise from a join. 
\end{rem}

\section{A Splitting Theorem for Extremal Sasakian Structures}\label{extsplitsect}
As in the K\"ahler case, the problem of finding an extremal {\it toric} Sasakian structure can be translated to finding an {\it extremal symplectic potential} which is a convex function on a certain polytope satisfying some boundary condition and a 4th order non-linear PDE. We now recall briefly the details of this correspondence.

\subsection{Extremal symplectic potential}
Recall from \S~\ref{conesect} that to any compact connected contact manifolds $(M^{2n+1},\cald)$ endowed with the contact action of a torus $\hat{\bbt}=\bbt^k$ and a fixed Reeb vector field $b \in \hat{\gt}= \mbox{Lie }\hat{\bbt}$, is associated a labelled polytope $$(P,\bu) = (P_b, \bu_b).$$  The toric case, as we assume it is for the rest of this secton, is when $\dim \hat{\bbt} = n+1$ and, in that case, the $\eta$-momentum map $\mu_\eta : M \rightarrow P$ is a quotient map. Consequently any $\hat{\bbt}$-invariant tensor on $M$ can be read off a corresponding tensor on $P$. This is explicit and very well understood for toric Sasaki metrics so that they correspond to {\it symplectic potentials} via the Hessian of the latter, see \cite{MaSp06} and also~\cite{Leg10}. This correspondence has first been developped in the context of K\"ahler geometry, by the work of Guillemin\cite{Gui94b}, Abreu\cite{Abr01}, Apostolov, Calderbank, Gauduchon, T\o nnesen-Friedman \cite{ACGT04}. To give more details, recall that we 
denote $P=\{x \in \mathcal{A}\, | \, l_i(x)\geq 0 \}$ where $\mathcal{A}$ is an $n$--dimensional affine space often identified with $\R^n$ and the defining affine functions, uniquely determined by $P$ and $\bu$, are $$l_i(\cdot) =\langle \cdot, \vec{n_i} \rangle -\lambda_i$$ for $i=1,\dots,d$ where $d$ is the number of facets of $P$. A symplectic potential on $(P,\bu)$ can be written $u_f = u_0 + f \in C^0(\bar{P})$ where $f \in C^\infty(\bar{P})$ and $$u_0 =\frac{1}{2}\sum_{i=1}^d l_i \log l_i,$$
satisfy the following Guillemin boundary conditions :

\begin{itemize}

\item $u_f$ is a smooth, strictly convex function on the interior of $P$;

\item when restricting to the interior of each face of $P$, $u_f$ is also a smooth, strictly convex function.

\end{itemize}

We denote the set of all relative symplectic potentials as
\begin{align*}
\mathcal{H}_S = \{ f \in C^\infty(\bar{P}) ~|~ & u_f = u_0 + f ~\mathrm{and} ~ u_f \mathrm{~satisfies ~the}\\
&  \mathrm{~Guillemin ~boundary ~conditions.} \}
\end{align*}

As recalled in Proposition \ref{bg00b}, whenever the Reeb vector field is quasi-regular then $(P,\bu)$ is rational, see definition \ref{Rational_Cone}, and is associated to a toric symplectic orbifold 
$$(N,\omega, \bbt) = (M/S_b, d\eta_b,  \hat{\bbt}/S_b)$$
via the Delzant--Lerman--Tolman \cite{LeTo97} correspondence which happens to be the quotient by $S_b\simeq S^1$ induced by the Reeb vector field. The set of compatible toric K\"ahler metrics on this symplectic orbifold is parametrized by $\mathcal{H}_S$ as well.       

On the other hand, suppose $\omega$ is a K\"ahler form invariant under the torus action $\TT^n$. Then the set of all $\TT^n$ invariant relative K\"ahler potentials is
$$
\mathcal{H}_K = \{ \varphi \in C_T^\infty(N) ~|~ \omega_\varphi = \omega + \sqrt{-1} \partial \bar{\partial} \varphi  > 0. \},
$$
where $C_T^\infty(N)$ is the set of all smooth, $\TT^n$ invariant functions on $N$. Thus $\mathcal{H}_S$ is in one-to-one correspondence to $\mathcal{H}_K$ through the Legendre transformation.

Let us consider $N^0 = \mu^{-1}(P)$. In fact, $N^0 = (\CC^*)^n = \RR^n \times \TT^n$. For any $\TT^n$ invariant K\"ahler metric $\omega$, we can express $\omega$ on $N^0$ as
$$
\omega = \sqrt{-1} \psi_{ij} d z^i \wedge d \bar{z}^j,
$$
where $z_i = \xi_i + t_i, ~ \xi_i \in \RR^n$ and $t_i \in \TT^n$; $\psi$ is a smooth, strictly convex function on $\RR^n$ and $\psi_{ij} = \frac{\partial^2 \psi}{\partial \xi_i \xi_j}$.
The scalar curvature $R_\omega$ on $N^0$ can be expressed as
$$
R_\omega(z) = - \psi^{ij} (\log \det (D^2 \psi))_{ij} (\xi).
$$
The Legendre transformation gives the dual coordinate $x = \nabla \psi(\xi)$ on $P$ and the symplectic potential
$$
u(x) = \sum_{i=1}^n x_i \xi_i - \psi(\xi).
$$
By Abreu's formula, we have
$$
R_\omega = R_u = - \sum_{i j} u^{ij}_{~ij}. 
$$
By definition $\omega$ is an extremal (see Section \ref{extsassect}) K\"ahler metric if $\nabla R_\omega$ is a holomorphic vector field. It implies that $R_u$ is an affine function on $P$. In such a case, we call $u$ is an {\it extremal symplectic potential}.

There is an important integral formula discovered by Donaldson \cite{Don02} on toric manifolds, which is a particular case of an observation by Futaki and Mabuchi \cite{FuMa95}, which can be extended to labelled polytopes.

\begin{prop}
\label{integral}
Let $u \in \mathcal{H}_S$ and $f \in C^\infty(\bar{P})$, then
$$
\int_P R_u f ~ d \mu = 2 \int_{\partial P} f ~ d \sigma - \int_P u^{ij} f_{ij} ~ d \mu,
$$
where $d \mu$ is the standard Lebesgue measure on $\RR^n$ and $d \sigma$ is a multiple of the standard Lebesgue measure on each facet $P_i$ such that $\vec{n}_i \wedge d \sigma = - d \mu$.
\end{prop}

For a synplectic toric orbifold $(N, \omega, \TT^n)$, we can define the {\it extremal affine function} $R_E$ on $P$ as follows: for any affine function $f$ on $P$, we have
$$
\int_P f R_E ~ d \mu = 2 \int_{\partial P} f ~ d \sigma.
$$

By Proposition (\ref{integral}), it is easy to verify that $u_f$ is associated to an extremal K\"ahler metric on $N$ if and only if \begin{equation}\label{extremalEQ}
R_f \equiv R_E. 
\end{equation}

Then, an {\it extremal symplectic potential} is a symplectic potential $u_f$ with $f\in \mathcal{H}_S$ solving the $4$--th order PDE \eqref{extremalEQ}. Note that everything said in this subsection makes sense for transversal K\"ahler geometry with labelled polytopes, therefore it  extends to irregular toric Sasaki manifolds, see~\cite{Leg10}.

\subsection{Splitting}
Let $(N_1, [\omega_1], \TT^{n_1})$ and $(N_2, [\omega_2], \TT^{n_2})$ be two toric orbifolds. Let $P_i \subset \RR^{n_i}, ~ i = 1,2$ be the rational Delzant polytope of the moment map $\mu_i$. Then the product $(N=N_1 \times N_2, [\omega = \omega_1 + \omega_2], \TT^n = \TT^{n_1} \times \TT^{n_2})$ is also a toric orbifold whose rational Delzant polytope is $P = P_1 \times P_2 \subset \RR^n = \RR^{n_1 + n_2}$. And the new moment map on $N$ is $\mu = \mu_1 \times \mu_2$.  Let $u_1, u_2$ be symplectic potentials on $P_1, P_2$ respectively. Then $u = u_1 + u_2$ is a symplectic potential on $P$. We denote the set of relative symplectic potentials of $N_1, N_2, N$ by $\mathcal{H}_S(N_1), \mathcal{H}_S (N_2), \mathcal{H}_S (N)$ respectively: 
\begin{align*}
\mathcal{H}_S(N_1) = \{ f \in C^\infty(\bar{P_1}) ~|~ & u_f = u_1 + f ~\mathrm{and} ~ u_f \mathrm{~satisfies ~the}\\
&  \mathrm{~Guillemin ~boundary ~conditions.} \}
\end{align*}

\begin{align*}
\mathcal{H}_S(N_2) = \{ f \in C^\infty(\bar{P_2}) ~|~ & u_f = u_2 + f ~\mathrm{and} ~ u_f \mathrm{~satisfies ~the}\\
&  \mathrm{~Guillemin ~boundary ~conditions.} \}
\end{align*}

\begin{align*}
\mathcal{H}_S(N) = \{ f \in C^\infty(\bar{P}) ~|~ & u_f = u + f ~\mathrm{and} ~ u_f \mathrm{~satisfies ~the}\\
&  \mathrm{~Guillemin ~boundary ~conditions.} \}
\end{align*}

\begin{lemma}
\label{split}
For any $f \in \mathcal{H}_S(N)$. Let $x = (x_1, \ldots, x_{n_1})$ be a coordinate system on $P_1$ and $y = (y_1, \ldots, y_{n_2})$ be a coordinate system on $P_2$. Then
$$
f_1(x) = \frac{1}{Vol(P_2)} \int_{P_2} f(x,y) ~ dy, \quad f_2(y) = \frac{1}{Vol(P_1)} \int_{P_1} f(x,y) ~ dx
$$
are relative symplectic potentials on $P_1, P_2$ respectively.
\end{lemma}

\begin{proof}
Without loss of generality, we only need to show that $f_1(x) \in \mathcal{H}_S(N_1)$, i.e., $u_{f_1} = u_1 + f_1$ satisfies:

\begin{itemize}
\item $u_{f_1}$ is a smooth, strictly convex function on $P_1$.

\item When restricting to each face of $P_1$, $u_{f_1}$ is still a smooth, strictly convex function.
\end{itemize}
Without loss of generality, we only show that $u_{f_1}$ is a smooth, strictly convex function on $P_1$. It is easy to see that $u_{f_1}$ is a smooth function. Also for any $x \in P$ and any nonzero vector $\vec{v}$, 
$$
(D^2 u_{f_1}(x))(\vec{v}, \vec{v}) =  \frac{1}{Vol(P_2)} \int_{P_2} D^2(u_1(x) + f(x,y))(\vec{v}, \vec{v}) ~ dy.
$$
Since for any $y \in P_2$, $D^2(u_1(x) + f(x,y))(\vec{v}, \vec{v}) > 0$, we conclude that $(D^2 u_{f_1}(x))(\vec{v}, \vec{v}) > 0$.
\end{proof}

Next we define a subspace of $\mathcal{H}_S(N_1)$, denoted by $\mathcal{G}(f_1)$ as follows: 
$$\mathcal{G}(f_1) := \{ g_1 \in \mathcal{H}_S(N_1) ~ | ~ \int_{P_1} f_1 ~ dx = \int_{P_1} g_1 ~ d x. \}$$
Similarly, we define a subspace of $\mathcal{H}_S(N_2)$ as follows:
 $$\mathcal{G}(f_2) := \{ g_2 \in \mathcal{H}_S(N_2) ~ | ~ \int_{P_2} f_2 ~ dy = \int_{P_2} g_2 ~ d y. \}$$

\begin{lemma}
\label{min}
Let $f \in \mathcal{H}_S(N)$. We obtain $f_1 \in \mathcal{H}_S(N_1), ~ f_2 \in \mathcal{H}_S(N_2)$ as in Lemma (\ref{split}). Then
$$
\int_P (f(x,y) - f_1(x) - f_2(y))^2 ~ dx dy \le \int_P (f(x,y) - g_1(x) - g_2(y))^2 ~ dx dy, 
$$
for any $g_1(x) \in \mathcal{G}(f_1), ~g_2(x) \in \mathcal{G}(f_2).$ Moreover, the equality holds iff $f_1(x) \equiv g_1(x)$ and $f_2(y) \equiv g_2(y)$.
\end{lemma}

\begin{proof}
It is equivalent to prove that
\begin{align*}
&\int_P -2 f(x,y) (f_1(x) + f_2(y)) + f_1^2(x) + f_2^2(y) ~ dx dy \\
\le &\int_P -2 f(x,y) (g_1(x) + g_2(y)) + g_1^2(x) + g_2^2(y) ~ dx dy.\\
\Leftrightarrow & \\
&- (Vol(P_2) \int_{P_1} f_1^2(x) ~ dx + Vol(P_1) \int_{P_2} f_2^2(y) ~ dy) \\
\le & Vol(P_2) \int_{P_1} - 2 f_1(x) g_1(x) + g_1^2(x) ~ d x +\\
& Vol(P_1) \int_{P_2} -2 f_2(y) g_2(y) + g_2^2(y) ~ dy \\
\Leftrightarrow & \\
& 0 \le Vol(P_2) \int_{P_1} (f_1(x) - g_1(x))^2 ~ d x + Vol(P_1) \int_{P_2} (f_2(y) - g_2(y))^2 ~ dy.
\end{align*}
Thus we obtain the desired inequality, and it is clear that the equality holds iff $f_1(x) \equiv g_1(x)$ and $f_2(y) \equiv g_2(y)$.
\end{proof}

Let $f \in \mathcal{H}_S(N)$ be a relative symplectic potential such that $u_f = u + f$ is an extremal symplectic potential. By Lemma (\ref{split}), we obtain relative symplectic potentials $f_1(x), f_2(y)$. Moreover, we have:

\begin{prop}
\label{potential_split}
$u_{f_1} = u_1 + f_1, ~ u_{f_2} = u_2 + f_2$ are extremal symplectic potentials on $P_1, P_2$ respectively.
\end{prop}

\begin{proof}
Let $R_{f_1}(x)$ be the scalar curvature of $u_{f_1}(x)$ and $R_{E, 1}(x)$ be the extremal affine function on $P_1$. Similarly, we let $R_{f_2}(y)$ be the scalar curvature of $u_{f_2}(y)$ and $R_{E, 2}(y)$ be the extremal affine function on $P_2$. Notice that $R_f(x,y) = R_{E, 1}(x) + R_{E, 2}(y)$.  Then there exists an $\epsilon > 0$ such that for any $t \in [0, \epsilon)$, we have
$$
f_1(t, x) := f_1 - t (R_{f_1} - R_{E,1}) \in \mathcal{G}(f_1), ~ f_2(t, y) := f_2 - t (R_{f_2} - R_{E,2}) \in \mathcal{G}(f_2).
$$
Then by Lemma (\ref{min}) and Proposition (\ref{integral}), we have
\begin{align*}
0 \le & \frac{\partial}{\partial t} \large|_{t=0} \int_P (f(x,y) - f_1(t, x) - f_2(t, y))^2 ~ dx dy\\
= & 2 \int_P (f(x,y) - f_1(x) - f_2(y)) (R_f(x,y) - R_{f_1}(x) - R_{f_2}(y)) ~ dx dy \\
=& - 2 \int_P (f_{ij} - f_{1, ij} - f_{2, ij}) (u^{ij}_f - u^{ij}_{f_1} - u^{ij}_{f_2}) ~ d x dy \\
=& - 2 \int_P (u_{f, ij} - u_{f_1, ij} - u_{f_2, ij}) (u^{ij}_f - u^{ij}_{f_1} - u^{ij}_{f_2}) ~ d x dy \\
\end{align*}

Let $v(x,y) = u_{f_1}(x) + u_{f_2}(y)$, then
\begin{align*}
& \int_P (u_{f, ij} - u_{f_1, ij} - u_{f_2, ij}) (u^{ij}_f - u^{ij}_{f_1} - u^{ij}_{f_2}) ~ d x dy\\
= & \int_P (u_{f, ij} - v_{ij}) (u^{ij}_f - v^{ij}) ~ d x dy\\
\le & 0.
\end{align*}
The last inequality uses the fact that for any positive constant $a, ~ (1-a)(1-a^{-1}) \le 0$. Moreover, the equality holds iff $(u_{f, ij}) \equiv (v_{ij})$. Thus we conclude that $f(x,y) = f_1(x) + f_2(y) + l_1(x) + l_2(y)$, where $l_1(x), l_2(y)$ are affine functions on $P_1, P_2$ respectively. Hence $u_{f_1}, u_{f_2}$ are extremal symplectic potentials on $P_1, P_2$ respectively.
\end{proof}

\begin{proof}[Proof of Theorem \ref{SasExtSPLIT}]
Here we adapt the arguments in \cite{Hua13}. Let $(M_1^{2*n_1+1}, \mathcal{S}_1, \TT^{n_1}),~ (M_1^{2*n_1+1}, \mathcal{S}_1, \TT^{n_1})$ be two quasi-regular Sasaki toric manifolds with moment maps $\mu_1, \mu_2$, respectively. Then we obtain two rational Delzant polytopes $P_1, P_2$ respectively. $M_3 = M_1 \star_{l_1, l_2} M_2$ is also a toric Sasaki manifold with a toric action $\TT^{n_3} = \TT^{n_1} \times \TT^{n_2}$. Its rational Delzant polytope $P_3 = l_1 P_1 \times l_2 P_2$. By hypothesis, there exists a transversally extremal K\"ahler metric in the transverse K\"ahler class of $M_3$. This implies that there exists an extremal symplectic potential $u_f$ on $P_3$. By Proposition \ref{potential_split}, we conclude that there exists extremal potentials $u_{f_1}, u_{f_2}$ on $l_1 P_1, l_2 P_2$ respectively. One can easily verify that $\frac{1}{l_1} u_{f_1} (x/l_1), \frac{1}{l_2} u_{f_2} (x/l_2)$ are extremal potentials on $P_1, P_2$ respectively. Thus, both $M_1, M_2$ admit extremal Sasaki structures.
\end{proof}

\def\cprime{$'$} \def\cprime{$'$} \def\cprime{$'$} \def\cprime{$'$}
  \def\cprime{$'$} \def\cprime{$'$} \def\cprime{$'$} \def\cprime{$'$}
  \def\cdprime{$''$} \def\cprime{$'$} \def\cprime{$'$} \def\cprime{$'$}
  \def\cprime{$'$}
\providecommand{\bysame}{\leavevmode\hbox to3em{\hrulefill}\thinspace}
\providecommand{\MR}{\relax\ifhmode\unskip\space\fi MR }
\providecommand{\MRhref}[2]{%
  \href{http://www.ams.org/mathscinet-getitem?mr=#1}{#2}
}
\providecommand{\href}[2]{#2}

\end{document}